\documentclass[a4paper]{article}
\usepackage[utf8]{inputenc}
\usepackage{amsmath}
\usepackage{amsfonts}
\usepackage{amsthm}
\usepackage{url}
\usepackage{mathtools}
\usepackage{xcolor}
\usepackage{ amssymb }

\usepackage{authblk}

\usepackage{mathrsfs}
\usepackage{textcomp}

\usepackage[top=2cm, bottom=4.5cm, left=2.5cm, right=2.5cm]{geometry}

\newtheorem{lemma}{Lemma}[section]
\newtheorem{theorem}{Theorem}[section]
\newtheorem{corollary}{Corollary}[section]
\newtheorem{proposition}{Proposition}[section]
\newtheorem{example}{Example}[section]
\newtheorem{definition}{Definition}[section]
\newtheorem{remark}{Remark}[section]

\newcommand{\perm}{P}
\newcommand{\LL}{L}
\newcommand{\invA}{B}

\newcommand{\mass}{A_M}
\newcommand{\KK}{K}
\newcommand{\KKother}{\widetilde{K}}
\newcommand{\MM}{M}
\newcommand{\MMother}{\widetilde{M}}
\newcommand{\bbeta}{\beta'}
\newcommand{\ww}{z'}
\newcommand{\lap}{A_S}
\newcommand{\qq}{\widetilde{q}}
\newcommand{\base}{2}
\newcommand{\Binf}{B^{(\infty)}}
\newcommand{\Ainf}{A^{(\infty)}}
\newcommand{\stiffnessQ}{R}
\newcommand{\shift}{\mathfrak{s}}
\DeclareMathOperator{\rank}{rank}

\usepackage[hidelinks]{hyperref}

\title{Tensor rank bounds and explicit QTT representations for the inverses of circulant matrices}

\author[1,2]{Lev Vysotsky}
\author[1]{Maxim Rakhuba}
\affil[1]{HSE University, Pokrovsky Boulevard 11, Moscow, 109028 Russia}
\affil[2]{Marchuk Institute of Numerical Mathematics of the Russian Academy of Sciences, 119333 Moscow, Russia}

\date{}

\begin{document}

\maketitle
\begin{abstract}
    In this paper, we are concerned with the inversion of circulant matrices and their quantized tensor-train~(QTT) structure.
    In particular, we show that the inverse of a complex circulant matrix $A$, generated by the first column of the form $(a_0,\dots,a_{m-1},0,\dots,0,a_{-n},\dots, a_{-1})^\top$ admits a QTT representation with the QTT ranks bounded by $(m+n)$. 
    Under certain assumptions on the entries of $A$, we also derive an explicit QTT representation of $A^{-1}$. The latter can be used, for instance, to overcome stability issues arising when numerically solving differential equations with periodic boundary conditions in the QTT format.
\end{abstract}

\section{Introduction}	

Tensor-train (TT) decomposition~\cite{osel-tt-2011} is a nonlinear representation of multidimensional arrays (tensors) that in many cases leads to significant compression ratios while maintaining high approximation accuracy.
Notably, TT can also be applied to low-dimensional data. 
For example, a vector (one-dimensional array) from $\mathbb{C}^{2^\LL}$ can be reshaped into an element of $\mathbb{C}^{2 \times \dots \times 2}$ ($\LL$-dimensional tensor) and then TT decomposition becomes applicable.
This idea was proposed in~\cite{osel-2d2d-2010,khor-qtt-2011} and is known under the name quantized TT (QTT) decomposition.
QTT decomposition has proven useful in various applications and in particular, for approximating functions and solving partial differential equations (PDEs)~\cite{khoromskij2018tensor}.

QTT decomposition can also be applied to linear operators in the form of matrices.
This is essential for constructing solvers for linear systems when the right-hand side is given in the QTT format and the goal is to approximate the solution also in the QTT format with the desired accuracy. 
In this paper, we are interested in studying the QTT representation of the inverse of a band circulant matrix: 
\begin{equation}\label{eq:A}
	A=
	\begin{bmatrix}
a_0     & a_{-1} & \dots  & a_{-n} & 0      & \dots  & 0       & a_{m-1} & \dots  & a_1     \\
a_1     & \ddots &        &        & \ddots &        &         &         & \ddots & \vdots  \\
\vdots  &        & \ddots &        &        & \ddots &         &         &        & a_{m-1} \\
a_{m-1} &        &        & \ddots &        &        & \ddots  &         &        & 0       \\
0       & \ddots &        &        & \ddots &        &         & \ddots  &        & \vdots  \\
\vdots  &        & \ddots &        &        & \ddots &         &         & \ddots & 0       \\
0       &        &        & \ddots &        &        & \ddots  &         &        & a_{-n}  \\
a_{-n}  &        &        &        & \ddots &        &         & \ddots  &        & \vdots  \\
\vdots  & \ddots &        &        &        & \ddots &         &         & \ddots & a_{-1}  \\
a_{-1}  & \dots  & a_{-n} & 0      & \dots  & 0      & a_{m-1} & \dots   & a_1    & a_0    
	\end{bmatrix}\in\mathbb{C}^{N\times N},
	\end{equation}
	which we will also denote as $A = \mathsf{circ}(a_0,\dots, a_{m-1}, 0,\dots, 0, a_{-n},\dots, a_{-1})$.
We are concerned with obtaining accurate QTT rank bounds for $A^{-1}$ and its explicit QTT representation when $N = 2^\LL$.
We emphasize the fact that the considered QTT ranks of matrices are \emph{not related} to the standard matrix rank and, hence, small QTT ranks do not imply that the matrix under consideration is singular.
The QTT rank bounds of $A^{-1}$ can be useful, e.g., for obtaining rank bounds for the solution of linear systems with $A$, while the explicit QTT representation of $A^{-1}$ can be used for constructing efficient solvers.

To formally introduce the QTT decomposition of a matrix, let us first introduce QTT decomposition of a vector $x = \{x_i\}_{i=0}^{2^\LL-1}\in\mathbb{C}^{2^\LL}$. 
Let us represent $x$ as a multidimensional array $X = \{X_{i_1\dots i_\LL}\}_{i_1,\dots, i_\LL=0}^{1,\dots,1}\in\mathbb{C}^{2\times \dots \times 2}$ by the following bijection between an integer $i=0,\dots,2^\LL-1$, and $\LL$ binary indices $(i_1,\dots,i_\LL)$: 
\[
    i =\overline{i_\LL \dots i_1} \equiv \sum_{k = 1}^\LL 2^{\LL - k} i_k,
\]
which is similar to a binary representation of $i$.
Then we apply the TT decomposition to $X$:
\[
    x_{\overline{i_\LL \dots i_1}} \equiv X_{i_1\dots i_\LL} = \sum_{\alpha_1,\dots,\alpha_{\LL-1}=1}^{r_1,\dots,r_{\LL-1}} G_{i_1 \alpha_1}^{(1)} G_{\alpha_1 i_2 \alpha_2}^{(2)} \dots G_{\alpha_{\LL-2} i_{\LL-1} \alpha_{\LL-1}}^{(\LL-1)} G_{\alpha_{\LL-1} i_\LL}^{(\LL)},
\]
where the minimal values of $r_1,\dots,r_{\LL-1}$ are called TT-ranks.
Storing the so-called \emph{core tensors}~$G^{(k)}$, $k=1,\dots,\LL$ requires $\mathcal{O}(\LL r^2)$ bytes (here $r=\max_k r_k$). 
We note that the total storage depends linearly on $\LL$ (assuming $r$ is independent of $\LL$) and, hence, logarithmically on the vector size $N = 2^\LL$. For function-related vectors, one can obtain bounds on the rank $r$, see~\cite{khoromskij2018tensor} and the reference therein. 

Similarly, we introduce QTT decomposition of a matrix $B = \{B_{i,j}\}_{i,j=0}^{2^\LL-1} \in\mathbb{C}^{2^\LL\times 2^\LL}$, by binarizing two indices: $i =\overline{i_\LL \dots i_1}$, $j =\overline{j_\LL \dots j_1}$ and merging $i_k,j_k$ into pairs: 
\begin{equation}\label{eq:tt_mat}
    B_{\overline{i_\LL \dots i_1}, \overline{j_\LL \dots j_1}} = 
    \sum_{\alpha_1,\dots,\alpha_{\LL-1}=1}^{r_1,\dots,r_{\LL-1}}
    G_{i_1 j_1 \alpha_1}^{(1)} G_{\alpha_1 i_2j_2 \alpha_2}^{(2)} \dots G_{\alpha_{\LL-2} i_{\LL-1}j_{\LL-1} \alpha_{\LL-1}}^{(\LL-1)} G_{\alpha_{\LL-1} i_\LL j_\LL}^{(\LL)}.
\end{equation}
As an example, one can consider an identity matrix $I\in\mathbb{C}^{2^\LL\times 2^\LL}$, whose elements can be expressed in terms of the Kronecker delta $\delta_{\alpha \beta}$ as follows:
\[
    I_{\overline{i_\LL \dots i_1}, \overline{j_\LL \dots j_1}} = \delta_{i_1 j_1} \delta_{i_2 j_2} \dots \delta_{i_\LL j_\LL},
\]
i.e., without any summation. Hence, the QTT ranks of the identity matrix are all equal to $1$, even though the matrix is of full rank.

Asymptotically, representation~\eqref{eq:tt_mat} leads to the same number of bytes in the core tensors $G^{(k)}$, $k=1,\dots,\LL$ as for the QTT representation of a vector: $\mathcal{O}(\LL r^2)$.
Fortunately, matrices arising after discretization of PDEs are also often of low rank~\cite{khkaz-lap-2012}.
Having access to both matrices and vectors in the QTT format, one can construct efficient algorithms for solving, for example, linear systems that avoid forming full matrices and vectors (see, e.g.~\cite{ds-amen-2014}).

To derive the QTT rank bounds of $B=A^{-1}$, where $A$ is as in~\eqref{eq:A}, we show that the elements of its first column $b$ have the form (Section~\ref{sec:circ}):
\begin{equation}\label{eq:first_col}
     b_{i} = \sum_{k=1}^s P_{k} (i) z_k^i,
\end{equation}
where $z_k$ are the roots of $g(z)=0$ and $h(z)=0$:
\begin{equation}\label{eq:g-and-h}
    g(z) = \sum_{k=-n}^{m-1}a_{k}z^{k+n}, \quad h(z) = \sum_{k=-n}^{m-1}a_{k}z^{m-k-1},
\end{equation}
located inside $U=\{z:|z|<1\}$, and where $P_k(i)$ is a certain polynomial of $i$ with the degree less than the multiplicity of $z_k$. 
We note that in~\cite{fuyong2011inverse}, the same formulas were obtained, but for the roots with all multiplicities equal to $1$ and in, e.g.,~\cite{searle1979inverting} multiplicities greater than $1$ were considered, but only for the case $m=2$, $n=1$. 
To overcome these limitations, we have generalized the result to the case of arbitrary multiplicities.
We impose only one restriction on $A$ that is fundamental to the proposed approach: the polynomials $g(z)$ and $h(z)$ from~\eqref{eq:g-and-h} must not have roots with absolute value 1, as it happens for singular matrices (but not only for them).

The structure of~\eqref{eq:first_col} is utilized to estimate the QTT ranks of $B = A^{-1}$, which appear (Section~\ref{sec:qtt_ranks}) to be bounded by $(m+n)$.
As an alternative, one may derive the QTT representation of $b$ and apply the result from~\cite{khkaz-conv-2013} to generate a QTT representation of a circulant matrix from its first column in the QTT format.
Nevertheless, we note that such an approach leads to overestimated QTT rank values of $B$.
The developed techniques are  applied to several examples of circulant matrices (Section~\ref{sec:examples}), including the case of pseudoinverses.

In the case of simple roots $z_k$, we additionally derive explicit formulas for the QTT representation of~$B$ (Section~\ref{sec:qtt_repr}). 
Finally, we test the stability of our formulas numerically (Section~\ref{sec:num_exp}) on the example of a one-dimensional convection-reaction-diffusion boundary value problem with periodic boundary conditions.
The numerical results suggest that we can apply the proposed explicit formulas for large values of $L$ without any stability issues.
This is by contrast to naively applying TT solvers for linear systems directly to the matrix $A$, explicitly assembled in the QTT format.

\emph{Related work.} For the QTT approximation of function-related vectors, we mention~\cite{khor-qtt-2011,dk-qtt-tucker-2013,gras-tenz-2010,vysotsky2021tt}. 
The techniques for deriving explicit QTT representations of QTT matrices were developed in~\cite{khkaz-lap-2012} and applied to specific matrices, arsing in discretization of the Laplacian operator on a uniform grid.
In~\cite{khkaz-lap-2012}, there were also provided the inverses of these matrices in special cases of Neumann and Dirichlet boundary conditions.
In the case of a Fourier matrix, no low-rank QTT representation exists, but the matrix-vector product can still be approximated efficiently in the QTT format~\cite{dks-ttfft-2012}.
In~\cite{khkaz-conv-2013}, the QTT rank bounds and explicit formulas were derived for multilevel Toeplitz and circulant matrices.
The rank bounds for band Toeplitz matrices were obtained in~\cite{otz-teninv-2011}. 

In~\cite{cor-eccomas-2016,kazeev2018quantized,cor-robqtt-2016pre}, it was observed that the straightforward application of TT optimization-based solvers to linear systems arising from PDEs with matrices in the QTT format, leads to severe numerical instabilities.
This problem was formalized in~\cite{bachmayr2018stability} and originates from both ill conditioning of discretized differential operators and the ill conditioning of the tensor representations themselves.
To overcome these issues, an explicit QTT representation of BPX-preconditioned systems was proposed in the same work, which was later used for multiscale and singularly-perturbed problems in~\cite{kazeev2020quantized,marcati2020low}.
In~\cite{rakhuba2021robust}, a robust and efficient solver based on the alternating direction implicit method (ADI) and explicit inversion formulas for tridiagonal Toeplitz matrices was developed.
This solver was applied to three-dimensional Schroedinger-type eigenvalue problems~\cite{marcati2019tensor}.

To the best of our knowledge, no QTT rank bounds or explicit QTT formulas were derived either for inverses of general band circulant matrices or their special cases, such as one-dimensional Laplacian discretization with the periodic boundary conditions.

\section{Circulant matrix inverse} \label{sec:circ}

In this section, we derive formulas for the inverse of a band circulant matrix, without imposing a QTT structure.
The main results of this section are Theorem~\ref{thm:general-inverse} and Corollary~\ref{cor:simple-roots}.
This section is mostly based on~\cite{fuyong2011inverse}, but we also take into account multiplicities of polynomial roots.

	Let us consider a nondegenerate circulant matrix $A \in \mathbb{C}^{N\times N}$ of the form~\eqref{eq:A} with the additional assumption that
	\[
	m \ge 1,~n \ge 0,~a_{m-1} \neq 0,~a_{-n} \neq 0.
	\]
	Let $B \in \mathbb{C}^{N \times N}$ denote the inverse of $A$: $B \equiv A^{-1}$.
	It is well-known~\cite{tyrtyshnikov1997brief} that the inverse of a circulant matrix is also a circulant.
	For $j = 0, \dots, N-1$, let $b_{j}$ be the $j$-th element of the first column of $B$, i.e., $B_{j, 0}$.
	Using definition of the inverse, we may write for all $k,\ell \in \{0,\dots,N-1\}$:
	\[
	\sum_{j=0}^{N-1} A_{k, j}B_{j,\ell}
	=
	\delta_{k,\ell}
	\equiv
	\begin{cases}
		1, &k=\ell, \\
		0, &\text{otherwise}.
	\end{cases}
	\]
	For circulants $A$ and $B$ this system of equations is equivalent to 
	\[
	\sum_{j=0}^{N-1} A_{(k-j) \bmod N, 0}\, B_{(j-\ell)\bmod N, 0}
	=
	\delta_{k,\ell},~k,\ell \in \{0,\dots, N-1\}
	\]
	or 
	\begin{equation}\label{eq:inverse}
	    \sum_{j=0}^{N-1} A_{(k-j) \bmod N, 0}\,b_{(j-\ell)\bmod N}
	=
	\delta_{k,\ell},
	~k,\ell \in \{0,\dots,N-1\}.
	\end{equation}

	Next, we consider a biinfinite Toeplitz matrix $\Ainf$ with the elements
	\[
	\Ainf_{i,j}
	=
	\begin{cases}
	    a_{i-j},~&\text{if } -n \le i-j \le m-1,\\
	    0,~&\text{otherwise}.
	\end{cases}
	\]
	In other words,
	\[
	\Ainf = \begin{bmatrix}
	\ddots & \ddots &        & \ddots  & \ddots  &        &        \\
	\ddots & a_0    & a_{-1} & \dots   &  a_{-n} &    0   & \dots  \\
	       & a_1    & \ddots & \ddots  &         & \ddots & \ddots \\
	\ddots & \vdots & \ddots &         &         &        &        \\
	\ddots & a_{m-1}&        &         &         &        &        \\
	       & 0      & \ddots &         &         &        &        \\
	       & \vdots & \ddots &         &         &        &        
	\end{bmatrix}.
	\]
	Consider the equation
	\begin{equation}\label{eq:infinite}
	\Ainf \xi = \beta,
	\end{equation}
	where $\xi$ and $\beta$ are biinfinite vectors with the elements
	$\xi_{j} = b_{j \bmod N}$ and 
	\[
	\beta_{j} = \begin{cases}
	    1,~\text{if}~j \bmod N = 0, \\
	    0,~\text{otherwise}.
	\end{cases}
	\]
	The notation $\Ainf\xi$ implies biinfinite matrix-by-vector multiplication:
	\[
	\beta_{k} = \sum_{j=-\infty}^{\infty}\Ainf_{k,j}\xi_{j},~~ k \in \mathbb{Z}.
	\]
	Note that these series are not truly infinite, as there are no more than $m+n$ nonzero elements in each row of $\Ainf$.
	Thus, each of these series is convergent.
	We can also rewrite equation~\eqref{eq:infinite} in a more verbose and, possibly, comprehensible form:
	\[
	\Ainf
	\begin{bmatrix}
	\vdots \\
	b_0 \\
	b_1 \\
	\vdots \\
	b_{N-1} \\
	b_0 \\
	\vdots
	\end{bmatrix}
	=
	\begin{bmatrix}
	\vdots \\
	1 \\
	0 \\
	\vdots \\
	0 \\
	1 \\
	\vdots
	\end{bmatrix}.
	\]
	
	\begin{lemma}\label{lm:equivalent}
	Equations~\eqref{eq:inverse} and~\eqref{eq:infinite}, considered as equations for $b_0,\dots,b_{N-1}$ are equivalent.
	\end{lemma}
	\begin{proof}
	    See the proof in Appendix~\ref{app:circ}.
	\end{proof}
	
	We will denote by $U$ the unit circle on the complex plane, i.e. $U = \{z\in\mathbb{C}: |z| = 1\}$.
	Let us consider the Laurent polynomial $f(z)$:
	\begin{equation}\label{eq:laurent}
	f(z) \equiv a_{-n}z^{-n} + \dots + a_{-1}z^{-1} + a_0 + a_1 z + \dots +a_{m-1}z^{m-1}.
	\end{equation}
	Let us additionally assume that $f(z)$ does not have roots on $U$ (i.e. $f(z) \neq 0$ for all $z \in U$).
	Note that this implies the same property for Laurent polynomial $f(z^{-1})$, as if $f(z_-^{-1}) = 0$ for some $z_- \in U$, then $f(z_+) = 0$ for $z_+ = z^{-1}_- \in U$.
	Now consider biinfinite matrix $\Binf$ with the elements:
	\begin{equation}\label{eq:B}
	\Binf_{j,\ell} = \frac{1}{2\pi i} \oint_U \frac{z^{\ell - j-1}dz}{f(z)}.
	\end{equation}
	\begin{lemma} \label{lm:unnamed}
	Matrix $\Binf$ is the right inverse of $\Ainf$:
	\[
	\Ainf \Binf = I^{(\infty)},
	\]
	where $I^{(\infty)}$ is a biinfinite identity matrix: $I^{(\infty)}_{k,\ell} = \delta_{k,\ell}$.
	\end{lemma}
	\begin{proof}
	    See the proof in Appendix~\ref{app:circ}.
	\end{proof}
	Now we can prove the main result of this section.
	\begin{theorem}\label{thm:general-inverse}
	Let $m$ and $n$ be nonnegative integers such that $m \ge 2$ and $A \in \mathbb{C}^{N\times N}$ be the circulant matrix of the form~\eqref{eq:A}.
	Denote by $g(z)$ and $h(z)$ the polynomials
	\[
	g(z) \equiv \sum_{k=-n}^{m-1}a_{k}z^{k+n},~~h(z) \equiv \sum_{k=-n}^{m-1}a_{k}z^{m-k-1}.
	\]
	Assume that $g(z)$ does not have roots on the unit cirle $U$.
	Denote $z_1, \dots, z_s$ the roots of $g(z)$ located inside $U$ and $p_1, \dots, p_k$ their respective orders.
	Similarly, denote $w_1, \dots, w_t$ the roots of $h(z)$ located inside $U$ and $q_1, \dots, q_t$ their respective orders.
	
	Under these conditions $A$ is invertible and its inverse is the circulant matrix $B\in\mathbb{C}^{N\times N}$ with the elements $B_{j,\ell} = b_{(j-\ell)\bmod N}:$
	\[
	b_j
	=
	\sum_{k=1}^s\sum_{p'=0}^{p_k-1}
	c_{g,k,p'}
	(-j+n-1+N)^{\underline{p}'} z_k^{-j+n-1+N-p'}
	+
	\sum_{k=1}^t\sum_{q'=0}^{q_k-1}
	c_{h,k,q'}
	(j+m-2)^{\underline{q}'}w_k^{j+m-2-q'},
	\]
	where
	\begin{align*}
    c_{g,k,p'}
	    &=
    \sum_{p=p'}^{p_k-1}\frac{1}{(p_k - 1)!}\binom{p_k-1}{p}\binom{p}{p'}
	\left(\frac{1}{g_k(z)}\right)^{(p_k-1-p)}_{\big|{z=z_k}}
	\left(\frac{1}{1-z^{N}}\right)^{(p-p')}_{\big|{z=z_k}},
	&
	g_k(z) = \prod_{m\neq k} (z-z_m)^{p_m},
	\\
	c_{h,k,q'}
	    &=
	\sum_{q=q'}^{q_k-1}
	\frac{1}{(q_k - 1)!}\binom{q_k-1}{q}\binom{q}{q'}
	\left(\frac{1}{h_k(z)}\right)^{(q_k-1-q)}_{\big|{z=w_k}}
	\left(\frac{1}{1-z^{N}}\right)^{(q-q')}_{\big|{z=w_k}},
	&
	h_k(z) = \prod_{m\neq k} (z-w_m)^{p_w},
	\end{align*}
	where $\left.(f(z))^{(p)}\right|_{z=w}$ denotes the $p$-th derivative of $f(z)$ at $z=w$
	and
	$M^{\underline{r}}$ denotes the falling factorial:
	\[
	M^{\underline{r}} = 
	\begin{cases}
	1, & \text{if } r = 0, \\
	M(M-1)\dots(M-r+1), & \text{otherwise}.
	\end{cases}
	\]
	\end{theorem}
	
	\begin{proof}
	Note that the Laurent polynomial $f(z)$ corresponding to $A$ does not have roots on $U$, because $g(z) = f(z)z^n$ by the Theorem's condition does not have such roots.
	Thus, the matrix $\Binf$ is defined correctly.
	Moreover, it means that $h(z)$ also does not have roots on $U$.
	
	Let us express the elements of $\Binf$ through the roots of $g(z)$ and $h(z)$.
	As $\Binf$ is a (biinfinite) circulant, it suffices to compute only its first column.
	First, let us perform the substitution $w = z^{-1}$ in the integral~\eqref{eq:B}:
	\[
	\Binf_{j,0}
	=
	\frac{1}{2\pi i} \oint_U \frac{z^{-j-1}dz}{f(z)}
	=
	-\frac{1}{2\pi i} \oint_U \frac{w^{j+1}}{f(w^{-1})}\cdot\frac{-dw}{w^2}
	=
	\frac{1}{2\pi i} \oint_U \frac{w^{j-1}dw}{f(w^{-1})}.
	\]
	Note that there appeared two minuses (one from the differential $d(w^{-1})$ and one from the change of the integral direction)  that gave a plus.
	Now we can split the formula for $\Binf_{j,0}$ into two cases:
	\[
	\Binf_{j,0} =
	\begin{dcases}
	\frac{1}{2\pi i} \oint_U\frac{z^{-j-1}dz}{f(z)}
	=
	\frac{1}{2\pi i} \oint_U\frac{z^{-j+n-1}dz}{g(z)}, & \text{if }j < 0, \\
	\frac{1}{2\pi i} \oint_U\frac{z^{j-1}dz}{f(z^{-1})}
	=
	\frac{1}{2\pi i} \oint_U\frac{z^{j+m-2}dz}{h(z)},  & \text{if } j \ge 0.
	\end{dcases}
	\]
	Note that $g(0) = a_{-n} \neq 0$ and $h(0) = a_{m-1} \neq 0$, so zero is not a root of $g(z)$ and $h(z)$.
	Moreover, as $m \ge 2$ and $n \ge 0$, both powers $j + m - 2$ and $-j + n - 1$ are nonnegative for the corresponding values of $j$.
	Thus the integrands above have singularities only at the roots of $g(z)$ and $h(z)$ respectively.
	We will transform the expression for $j < 0$, as the case of $j \ge 0$ is handled analogously.
	
	Using the residue theorem and the formula for the residue at the pole of order $p_k$, we can write (for $j < 0$):
	\[
	\Binf_{j,0} =
	\frac{1}{2\pi i} \oint_U\frac{z^{-j+n-1}dz}{g(z)} = \sum_{k=1}^s \mathrm{Res}\left(\frac{z^{-j+n-1}}{g(z)}, z_k\right) 
	=
	\sum_{k=1}^s\frac{1}{(p_k - 1)!}\left(\frac{z^{-j+n-1}}{g_k(z)}\right)^{(p_k-1)}_{\big|{z=z_k}}.
	\]
	Using the  higher order product rule, we obtain
	\[
	\Binf_{j,0}
	=
	\sum_{k=1}^s\frac{1}{(p_k - 1)!}\sum_{p=0}^{p_k-1}\binom{p_k-1}{p}
	(z^{-j+n-1})^{(p)}|_{z=z_k} \left(\frac{1}{g_k(z)}\right)^{(p_k-1-p)}_{\big|{z=z_k}}.
	\]
	For $j \ge 0$ the formula is very similar:
	\[
	\Binf_{j,0}
	=
	\sum_{k=1}^t\frac{1}{(q_k - 1)!}\sum_{q=0}^{q_k-1}\binom{q_k-1}{q}
	(z^{j+m-2})^{(q)}|_{z=w_k} \left(\frac{1}{h_k(z)}\right)^{(q_k-1-q)}_{\big|{z=w_k}}.
	\]
	
	Let us now demonstrate that $\xi = \Binf\beta$ is a solution of~\eqref{eq:infinite} and is $N$-periodic.
	First,
	\[
	\sum_{\ell=-\infty}^{\infty} \Binf_{j,\ell} \beta_{\ell}
	=
	\sum_{\ell=-\infty}^{\infty}\Binf_{j, N\ell}
	= 
	\sum_{\ell=-\infty}^{\infty}\Binf_{j-N\ell, 0}.
	\]
	The periodicity of this expression is obvious: if $j = j_1 + Nj_2$, we can write using the change of the summation variable:
	\[
	\sum_{\ell=-\infty}^{\infty}\Binf_{j-N\ell, 0}
	=
	\sum_{\ell=-\infty}^{\infty}\Binf_{j_1-N(\ell-j_2), 0}
	=
	\sum_{\ell=-\infty}^{\infty}\Binf_{j_1-N\ell, 0}.
	\]
	Thus, it is sufficient to consider the case $j = 0,\dots, N-1$. 
	For these values of $j$ we split the sum in the following way:
	\[
	\sum_{\ell=-\infty}^{\infty}\Binf_{j-N\ell, 0} 
	=
	\sum_{\ell=1}^{\infty}\Binf_{j-N\ell, 0}
	+
	\sum_{\ell=-\infty}^{0}\Binf_{j-N\ell, 0}.
	\]
	In the first sum the row index $j - N\ell$ is negative for all values of $\ell$ and in the second sum the row index is nonnegative for all values of $\ell$.
	Thus, we can use the formulas for $B_{j,0}$ obtained above to write
	\begin{align}
	\sum_{\ell=-\infty}^{\infty}B_{j-N\ell, 0} 
	&=
	\sum_{k=1}^s\frac{1}{(p_k - 1)!}\sum_{p=0}^{p_k-1}\binom{p_k-1}{p}
	\left(\frac{1}{g_k(z)}\right)^{(p_k-1-p)}_{\big|{z=z_k}}
	\sum_{\ell=1}^{\infty}
	(z^{-j+N\ell+n-1})^{(p)}|_{z=z_k}  \nonumber
		+\\&+
	\sum_{k=1}^t\frac{1}{(q_k - 1)!}\sum_{q=0}^{q_k-1}\binom{q_k-1}{q}
	\left(\frac{1}{h_k(z)}\right)^{(q_k-1-q)}_{\big|{z=w_k}}
	\sum_{\ell=-\infty}^{0}
	(z^{j-N\ell+m-2})^{(q)}|_{z=w_k}. \label{eq:sum-of-B}
	\end{align}
    As $|z_k| < 1$ and $|w_k| < 1$, the series
    $
    \sum_{\ell=1}^{\infty}
	P(\ell)z^{-j+N\ell+n-1}
	$
    and 
    $
    \sum_{\ell=0}^{\infty}
	P(\ell)z^{j+N\ell+m-2}
    $
	converge uniformly in the (small enough) neighbourhood of $z_k$ and $w_k$ respectively for any polynomial $P(\ell)$.
	Thus, the summation and the $p$-th derivative can be swapped, so we can obtain
	\[
	\sum_{\ell=1}^{\infty}
	(z^{-j+N\ell+n-1})^{(p)}|_{z=z_k}
	=
    \left(\frac{z^{-j+n-1+N}}{1-z^{N}}\right)^{(p)}_{\big|{z=z_k}}
    =
	\sum_{p'=0}^p\binom{p}{p'}(z^{-j+n-1+N})^{(p')}|_{z=z_k} \left(\frac{1}{1-z^{N}}\right)^{(p-p')}_{\big|{z=z_k}}.
	\]
	The other series is computed in the same manner:	
	\[
	\sum_{\ell=0}^{\infty}
	(z^{j+N\ell+m-2})^{(q)}|_{z=w_k}
	=
	\sum_{q'=0}^q\binom{q}{q'}(z^{j+m-2})^{(q')}|_{z=w_k} \left(\frac{1}{1-z^{N}}\right)^{(q-q')}_{\big|{z=w_k}}.
	\]
	Plugging these expression in~\eqref{eq:sum-of-B}, we finally get
	\begin{align*}
	\sum_{\ell=-\infty}^{\infty}B_{j-N\ell, 0} 
	&=
	\sum_{k=1}^s\sum_{p=0}^{p_k-1}\sum_{p'=0}^p\frac{1}{(p_k - 1)!}\binom{p_k-1}{p}\binom{p}{p'}
	\left(\frac{1}{g_k(z)}\right)^{(p_k-1-p)}_{\big|{z=z_k}}
	\left(\frac{1}{1-z^{N}}\right)^{(p-p')}_{\big|{z=z_k}}
	(z^{-j+n-1+N})^{(p')}|_{z=z_k}
	+\\&+
	\sum_{k=1}^t\sum_{q=0}^{q_k-1}\sum_{q'=0}^q
	\frac{1}{(q_k - 1)!}\binom{q_k-1}{q}\binom{q}{q'}
	\left(\frac{1}{h_k(z)}\right)^{(q_k-1-q)}_{\big|{z=w_k}}
	\left(\frac{1}{1-z^{N}}\right)^{(q-q')}_{\big|{z=w_k}}
	(z^{j+m-2})^{(q')}|_{z=w_k}
	\end{align*}
	Changing the summation order and using the formula for $c_{g,k,p'}$ and $c_{h,k,q'}$, we can write:
	\[
	\sum_{\ell=-\infty}^{\infty}\Binf_{j-N\ell, 0} 
	=
	\sum_{k=1}^s\sum_{p'=0}^{p_k-1}
	c_{g,k,p'}
	(z^{-j+n-1+N})^{(p')}|_{z=z_k}
	+
	\sum_{k=1}^t\sum_{q'=0}^{q_k-1}
	c_{h,k,q'}
	(z^{j+m-2})^{(q')}|_{z=w_k}
	\]
	Now, the $p'$-th derivative of a monomial can be written using the falling factorial:
	\[
	(z^{-j+n-1+N})^{(p')}|_{z=z_k} = (-j+n-1+N)^{\underline{p}'} z_k^{-j+n-1+N-p'},
	\]
	the similar holds for $(z^{j+m-2})^{(q')}|_{z=w_k}$.
	So finally we come to
	\[
	\sum_{\ell=-\infty}^{\infty}\Binf_{j-N\ell, 0} 
	=
	\sum_{k=1}^s\sum_{p'=0}^{p_k-1}
	c_{g,k,p'}
	(-j+n-1+N)^{\underline{p}'} z_k^{-j+n-1+N-p'}
	+
	\sum_{k=1}^t\sum_{q'=0}^{q_k-1}
	c_{h,k,q'}
	(j+m-2)^{\underline{q}'}w_k^{j+m-2-q'}.
	\]
	
	Note that we have implicitly shown that the series $\sum_{\ell}\Binf_{j, N\ell}$ converge and, therefore, the biinfinite vector $\xi = \Binf\beta$ is correctly defined (its $N$-periodicity has been shown above).
	Now it remains to demonstrate that $\xi$ is the solution of~\eqref{eq:infinite}:
	\[
	\Ainf\xi = \Ainf(\Binf\beta) = (\Ainf\Binf)\beta = I^{(\infty)}\beta  =\beta.
	\]
	We have used the fact that multiplication of biinfinite matrices $\Ainf$, $\Binf$ and $\beta$ is associative.
	Generally, such multiplication is not associative, but in our case multiplication by $\Ainf$ involves only finite number of summands for each element of the product, so the associativity property holds.
	Application of Lemma~\ref{lm:equivalent} finishes the proof.
	\end{proof}
	
	\begin{corollary}\label{cor:simple-roots}
	Under the conditions of Theorem~\ref{thm:general-inverse}, if both $g(z)$ and $h(z)$ have only simple roots inside the unit circle $U$, then $A$ is invertible and its inverse is the circulant matrix $B\in\mathbb{C}^{N\times N}$ with elements $B_{j,\ell} = b_{(j-\ell)\bmod N}:$
	\[
	b_j
	=
    \sum_{k=1}^s
    \frac{1}{g_k(z_k)(1-z_k^{N})}
	z_k^{-j+n-1+N}
	+
	\sum_{k=1}^t
	\frac{1}{h_k(w_k)(1-w_k^{N})}
	w_k^{j+m-2},
	\]
	where
	\[
	g_k(z) = \frac{g(z)}{z-z_k},~~~
	h_k(z) = \frac{h(z)}{z-w_k}.
	\]
	\end{corollary}
	
	We have proven that if $f(z)$ (or, equivalently, $g(z)$ or $h(z)$) does not have roots on $U$, then $A$ is invertible.
	The reverse, however, is not generally true, as is demonstrated by the following result and a counterexample.
	
	\begin{proposition}
	    The circulant $A \in \mathbb{C}^{N\times N}$ is invertible if and only if the corresponding polynomial $g(z)$ does not have roots of the form $e^{-\frac{2 \pi i}{N}s}$, $s \in \{0, \dots, N-1\}$.
	\end{proposition}
	\begin{proof}
	    It is well known that the eigenvalues $\lambda_s$ of $A$ are the elements of column $F_N A_{:,0}$, where $F_N \in \mathbb{C}^{N\times N}$ is the Fourier matrix: $(F_N)_{s,t} = e^{-\frac{2 \pi i}{N}st}$.
	    Thus, 
	    \begin{align*}
	    \lambda_s = \sum_{t=0}^{N-1}e^{-\frac{2 \pi i}{N}st} A_{t,0}
	    &=
	    \sum_{t=0}^{m-1}e^{-\frac{2 \pi i}{N}st} A_{t,0}
	    +
	    \sum_{t=-n}^{-1}e^{-\frac{2 \pi i}{N}s(N+t)} A_{N+t,0}
	    = \\ &=
	    \sum_{t=0}^{m-1}e^{-\frac{2 \pi i}{N}st} a_t
	    +
	    \sum_{t=-n}^{-1}e^{-\frac{2 \pi i}{N}st} a_t
	    =
	    f(e^{-\frac{2 \pi i}{N}s}).
	    \end{align*}
	   $A$ is invertible if and only if $\lambda_s \neq 0$ or, equivalently, $f(e^{-\frac{2 \pi i}{N}s}) \neq 0$ for all $s = 0, \dots, N-1$.
	   This property is equivalent to the statement that $g(z) = z^nf(z)$ does not have roots of the described form.
	\end{proof}
	\begin{example}
	    Let us also construct a counterexample. Consider the following circulant:
	\[
	A = \begin{bmatrix}
	    1 & 0 & 1 \\
	    1 & 1 & 0 \\
	    0 & 1 & 1
	\end{bmatrix}.
	\]
	It is invertible:
	\[
	A^{-1} = \frac{1}{2}
	\begin{bmatrix}
	    1 & 1  & -1 \\
	    -1& 1  & 1 \\
	    1 & -1 & 1
	\end{bmatrix}.
	\]
	But $f(z) = 1 + z$ has a root $(-1) \in U$, so Theorem~\ref{thm:general-inverse} is not applicable.
	\end{example}

\section{QTT rank bounds of circulants}\label{sec:qtt_ranks}

This section is devoted to the derivation of QTT rank bounds of the circulant  matrix inverse.
The following theorem gives us a general result for tensor rank bounds of circulants with a specific first column, which belongs to a low-dimensional space of discrete functions.

Note that in the current and following sections we use letter $i$ to denote indices instead of  imaginary unit in contrast to Section~\ref{sec:circ}. 
For the latter we will use the notation $\sqrt{-1}$.

	\begin{theorem}\label{thm:qtt-rank-general}
		Consider a function $f: \mathbb{Z} \to \mathbb{C}$ and let $f_q(i) \equiv f(i+q)$ for every fixed $q \in \mathbb{Z}$.
 		Assume that the following linear space of functions is finite-dimensional:
		\[
		V \equiv \mathrm{span}\{f_q~|~q \in \mathbb{Z}\}.
		\]
		Consider a circulant $A \in \mathbb{C}^{(N_1N_2) \times (N_1N_2)}$ with the elements $A_{ij} = f((i-j) \bmod N_1N_2)$, and a following ``reshaped'' matrix $\widehat{A} \in \mathbb{C}^{N_1^2 \times N_2^2}$:
		\[
		\widehat{A}_{i_1N_1 + j_1, i_2N_2 + j_2} = A_{i_1N_2 + i_2,j_1N_2 + j_2},~~ i_1, j_1 \in \{0, \dots, N_1-1\},~~i_2, j_2 \in \{0, \dots, N_2-1\}.
		\]
		Then
		\[
		\rank \widehat{A} \le 1 + \dim V.
		\]
	\end{theorem}
	\begin{proof}
	    Let us transform the formula for an element of  $\widehat{A}$:
	    \begin{align*}
			\widehat{A}_{i_1N_1 + i_1, i_2N_2 + i_2} &= 
			f\Big(\big((i_1N_2 + i_2) - (j_1N_2 + j_2)\big) \bmod N_1N_2\Big) = \\ &=
			f\Big(\big((i_1-j_1)N_2 + (i_2 - j_2)\big) \bmod N_1N_2\Big) = \\ &=
			f\big((\Delta_1 N_2 + \Delta_2) \bmod N_1N_2\big),
		\end{align*}
		where
    	\begin{align*}
		\Delta_1 &= \Delta_1(i_1,j_1) \equiv i_1-j_1, \\
		\Delta_2 &= \Delta_2(i_2,j_2) \equiv i_2-j_2.
		\end{align*}

		Note that $\Delta_1 \in [-N_1+1,N_1-1]$ and $\Delta_2 \in [-N_2+1, N_2 -1]$.
		Thus, $\Delta_1N_2 + \Delta_2  \in [-N_1N_2+1, N_1N_2 - 1]$, so
		\[
		(\Delta_1N_2 + \Delta_2) \bmod N_1N_2 =
		\begin{cases}
		    \Delta_1N_2 + \Delta_2,          & \text{ if } \Delta_1 > 0, \\
		    \Delta_1N_2 + \Delta_2 + N_1N_2, & \text{ if } \Delta_1 < 0, \\
		    \Delta_2 \bmod N_1,         & \text{ if } \Delta_1 = 0.
		\end{cases}
		\] 
		Now it can be seen that the row $u \equiv \widehat{A}_{i_1N_1+j_1,:}$, corresponding to  any $\Delta_1 \neq 0$ (i.e. $i_1 \neq j_1$) has the form
		\begin{equation}\label{eq:column}
		 	u_{i_2N_2+j_2} = 
		 	f\big(\Delta_2(i_2,j_2) + \varphi(\Delta_1)\big)
		\end{equation}
		for some $\varphi(\Delta_1) \in \mathbb{Z}$.
		Let us fix a basis $\{g^{(1)}, \dots, g^{(r)}\}$ of the function space  $V$.
		Each row $u$ of the form~\eqref{eq:column} can be expressed as a linear combination of columns $v^{(1)}, \dots, v^{(r)} \in \mathbb{C}^{N_2^2}$:
		\[
			v^{(i)}_{i_2N_2+j_2} = 
			g^{(i)}(\Delta_2(i_2,j_2)).
		\]
        		
		On the other hand, the rows of $\widehat{A}$ corresponding to $\Delta_1 = 0$ (i.e. $i_1 = j_1$) are all equal to each other and to the vector $v^{(r+1)} \in \mathbb{C}^{N_2^2}$ with the elements
		\[
		v^{(r+1)}_{i_2N_2+j_2} = f(\Delta_2(i_2, j_2) \bmod N_1).
		\]
		We have proven that $\mathrm{Im}(\widehat{A}) \subset \mathrm{span}\{v^{(1)}, \dots, v^{(r+1)}\}$.
		Thus, the rank of $\widehat{A}$ does not exceed $1 + \dim V$.
	\end{proof}
    \begin{corollary}
        If under the conditions of Theorem~\ref{thm:qtt-rank-general} 
        the circulant $A$ is of shape ${\base^\LL \times \base^\LL}$ for some positive integer $\LL$,
        then it admits a QTT representation with the ranks not greater than $1 + \dim V$.
    \end{corollary}
    \begin{proof}
    First, we recall the fact that the $k$-th QTT rank of $A$, $k = 1,\dots,\LL-1$, is equal to the rank of unfolding matrix $A_k \in \mathbb{C}^{\base^{2k} \times \base^{2(\LL-k)}}$ (see~\cite{osel-tt-2011}):
    \[
    (A_k)_{\overline{j_ki_k\dots j_1i_1},~\overline{j_{\LL}i_{\LL}\dots j_{k+1}i_{k+1}}} = A_{\overline{i_{\LL}\dots i_1},~\overline{j_{\LL}\dots j_1}}.
    \]
    Let us denote $N_1 \equiv \base^k$, $N_2 \equiv \base^{\LL-k}$ and consider the matrix $\widehat{A}_k \in \mathbb{C}^{N_1^2 \times N_2^2}$ from Theorem~\ref{thm:qtt-rank-general}:
    \[
    (\widehat{A}_k)_{\overline{i_k\dots i_1 j_k \dots j_1}, \overline{i_\LL\dots i_{k+1} j_\LL \dots j_{k+1}}} = A_{\overline{i_{\LL}\dots i_1},~\overline{j_{\LL}\dots j_1}}.
    \]
    From the Theorem it follows that $\rank \widehat{A} \le 1 + \dim V$.
    It remains to notice that $A_k$ can be obtained from $\widehat{A}_k$ by permuting its rows and columns.
    In other words, $A_k = P_1 \widehat{A}_k P_2$ where $P_1$ and $P_2$ are permutation matrices of appropriate size.
    Thus, $\rank A_k = \rank \widehat{A}_k \le \dim V + 1$.
    \end{proof}

	\begin{corollary}\label{cor:qtt-ranks-special-circulant}
		Let $A \in \mathbb{C}^{\base^\LL \times \base^\LL}$ be a circulant with elements $f((i-j) \bmod \base^\LL)$ where \[
			f(i) = \sum_{k=1}^s P_k(i) z_k^i
		\]
		for some polynomials $P_1(i),\dots,P_s(i)$ of degrees $p_1, \dots, p_s$ respectively, and $z_1, \dots, z_s \in \mathbb{C}$.
		Then QTT ranks of $A$ do not exceed $s + 1 + p_1 + \dots + p_s$.
	\end{corollary}
	\begin{proof}
		Note that
		\[
		f_j(i) = f(i+j) = \sum_{k=1}^s P_k(i+j)z_k^{i+j} = \sum_{k=1}^s P_{k,j}(i)z_k^{i}
		\]
		for some polynomials $P_{k,j}(i)$ of degrees $p_1,\dots,p_s$ respectively.
		Thus, the set of functions
		\[
		\{z_1^i,iz_1^i,\dots,i^{p_1}z_1^i,\dots,z_s^i,iz_s^i,\dots, i^{p_s}z_s^i\}
		\]
		contains the basis of space $V$, so $\dim V \le (1+p_1)+\dots+(1+p_s)$.
	\end{proof}
	\begin{corollary}
		Fix an arbitrary positive integer $\LL$ and let $A \in \mathbb{C}^{\base^\LL \times \base^\LL}$ be a circulant satisfying the conditions of Theorem~\ref{thm:general-inverse}.
		Then the QTT ranks of $A^{-1}$ do not exceed $m+n$.
	\end{corollary}
	\begin{proof}
        From Theorem~\ref{thm:general-inverse} it follows that $(A^{-1})_{ij} = f((i-j) \bmod \base^\LL)$, where 
        \[
        f(i) = 
    	\sum_{k=1}^s\sum_{p'=0}^{p_k-1}
    	c_{g,k,p'}
    	(-i+n-1+N)^{\underline{p}'} z_k^{-i+n-1+N-p'}
    	+
    	\sum_{k=1}^t\sum_{q'=0}^{q_k-1}
    	c_{h,k,q'}
    	(i+m-2)^{\underline{q}'}w_k^{i+m-2-q'}.
        \]
        We can rewrite this equality in the following way:
        \begin{align*}
        f(i) &= 
    	\sum_{k=1}^sP_k(i)\left(\frac{1}{z_k}\right)^{i}
    	+
    	\sum_{k=1}^tQ_k(i)w_k^{i},\\
    	P_k(i) &= 
    	\sum_{p'=0}^{p_k-1}
    	c_{g,k,p'}
    	z_k^{n-1+N-p'}
    	(-i+n-1+N)^{\underline{p}'}, \\
        Q_k(i) &=
    	\sum_{q'=0}^{q_k-1}
    	c_{h,k,q'}
    	w_k^{m-2-q'}
    	(i+m-2)^{\underline{q}'}.
        \end{align*}
        Obviously, $P_k(i)$ and $Q_k(i)$, viewed as functions of $i$, are polynomials of degree $p_k-1$ and $q_k-1$ respectively.
        From Corollary~\ref{cor:qtt-ranks-special-circulant} it follows that QTT ranks of $A^{-1}$ do not exceed
        \begin{equation}\label{eq:rank}
        1 + \sum_{k=1}^s p_k + \sum_{k=1}^t q_k.    
        \end{equation}
        
        Consider all roots of $g(z)$:  $z_1, \dots, z_s, z_{s+1}, \dots, z_{s'}$ with their respective multiplicities: $p_1, \dots, p_s$, $p_{s+1}, \dots, p_{s'}$.
        Here $z_k$ lies inside the unit circle $U$ for $k \le s$ and outside of it for $s < k \le s'$.
        Next we note that $h(z) = g(z^{-1})z^{m+n-1}$.
        Together with the facts that $\deg h(z) = \deg g(z)$ and $a_{m-1} \neq 0$ and $a_{-n} \neq 0$ it implies that the roots of $h(z)$ are $z_1^{-1}, \dots, z_s^{-1}, z_{s+1}^{-1}, \dots, z_{s'}^{-1}$ (with respective multiplicities $p_1, \dots, p_s, p_{s+1}, \dots, p_{s'}$).
        Thus we can conclude that $s' = s + t$ and for some permutation $\sigma \in S_t$ we have $w_{k} = z_{s+\sigma(k)}^{-1}$ and $q_k = p_{s+\sigma(k)}$.
        
        As the sum of multiplicities of roots of a polynomial of degree $(m + n -1)$   equals $(m+n-1)$, we can use~\eqref{eq:rank} to state that QTT ranks of $A^{-1}$ do not exceed
        \[
            1 + \sum_{k=1}^s p_k + \sum_{k=1}^t p_{s+\sigma(k)}
            =
            1 + \sum_{k=1}^{s'} p_k = 1 + (m+n-1) = m+n. 
        \]
	\end{proof}

\section{Explicit QTT representation} \label{sec:qtt_repr}

To derive an explicit QTT representation of a matrix, it is convenient to introduce the so-called strong Kronecker product~\cite{khkaz-lap-2012}.
Before we define it, let us introduce the \emph{core matrices} $Q_k$, associated with the $k$-th core $G^{(k)}$, $k=1,\dots,\LL$, as follows:
\begin{equation}\label{eq:core_mat}
    Q_k = 
    \begin{bmatrix}
        G^{(k)}(1, :, :, 1) & \dots & G^{(k)}(1, :, :, r_k) \\
        \vdots & \ddots & \vdots \\
        G^{(k)}(r_{k-1}, :, :, 1) & \dots & G^{(k)}(r_{k-1}, :, :, r_k)
    \end{bmatrix}
    \in\mathbb{C}^{2r_{k-1} \times 2r_k}.
\end{equation}
The strong Kronecker product, denoted as $\Join$, is defined for such block matrices. 
\begin{definition}
Let $A$ and $B$ be block matrices, both with $p\times q$ and $q\times r$ blocks $A_{\alpha\gamma}$, $B_{\gamma\beta}$ of the size $2\times 2$, where $\alpha=1,\dots,p$, $\beta=1,\dots,r$, $\gamma=1,\dots,q$. Their strong Kronecker product $A\Join B$ is a $p\times r$ block matrix with the blocks of the size $4\times 4$ such that:
\[
    (A\Join B)_{\alpha\beta} =\sum_{\gamma=1}^q A_{\alpha\gamma} \otimes B_{\gamma \beta}.
\]
\end{definition}
Now we can write the matrix $B$, given by its QTT cores $G^{(k)}$ and the respective core matrices $Q_k$ (Eq.~\eqref{eq:core_mat}) in terms of the strong Kronecker product as~\cite{khkaz-lap-2012}:
\[
    B = Q_1 \Join Q_2 \Join \dots \Join Q_\LL.
\]

\begin{lemma}[\cite{khkaz-lap-2012},\cite{osel-tt-2011}] \label{lm:sumtt}
Let $B_1 = Q_1^{(1)} \Join \dots \Join Q_\LL^{(1)}$ and $B_2 = Q_1^{(2)} \Join \dots \Join Q_\LL^{(2)}$. Then for $c_1,c_2\in\mathbb{C}$, the QTT representation of $c_1 B_1 + c_2 B$ can be written in terms of its core matrices as
\[
    c_1 B_1 + c_2 B_2 = 
    \begin{bmatrix}
        Q_1^{(1)} & Q_1^{(2)}
    \end{bmatrix}
    \Join 
    \begin{bmatrix}
        Q_2^{(1)} \\ & Q_2^{(2)}
    \end{bmatrix}
    \Join 
    \dots 
    \Join
    \begin{bmatrix}
        Q_{\LL-1}^{(1)} \\ & Q_{\LL-1}^{(2)}
    \end{bmatrix}
    \Join
    \begin{bmatrix}
        c_1 Q_\LL^{(1)} \\  c_2 Q_\LL^{(2)}
    \end{bmatrix}.
\]
\end{lemma}
A direct consequence of Lemma~\ref{lm:sumtt} is that the QTT ranks of a sum of two QTT matrices are bounded by the sum of the QTT ranks of the summands.

Next, let us move to the derivation of the explicit QTT representation of a circulant matrix inverse. 
Let us introduce a cyclic permutation matrix
\[
    \perm_\LL = 
    \begin{bmatrix}
        0 & & & 1 \\
        1 & \ddots \\
        & \ddots & \ddots \\
        & & 1 & 0
    \end{bmatrix}\in\mathbb{R}^{2^\LL \times 2^\LL},
\]
which allows us to naturally represent any circulant in terms of the powers of $\perm_\LL$:
\[
    \mathsf{circ}(c_0,c_1,\dots, c_{2^{\LL}-1}) = c_0 I + c_1 P_\LL + \dots + c_{2^\LL} P_\LL^{2^\LL -1}.
\]
The following explicit QTT representation holds for the $\overline{i_1\dots i_\LL}$-th power of $\perm_\LL$.

\begin{lemma}[\cite{khkaz-conv-2013}, Lemma 3.2] \label{lm:perm}
Let $\LL \geq 2$. Then $\perm_\LL^{\,\overline{i_1\dots i_\LL}}$ admits a QTT representation with the ranks $(2,2,\dots,2)$:
\[
    \perm_\LL^{\,\overline{i_1\dots i_\LL}} = U_{i_\LL} \Join V_{i_{\LL-1}} \Join \dots \Join V_{i_{2}} \Join W_{i_1},
\]
where
\[
\begin{split}
    &U_0 = 
    \begin{bmatrix}
        I & H
    \end{bmatrix},
    \quad
    U_1 = 
    \begin{bmatrix}
        H & I
    \end{bmatrix},
    \\
    &V_0 = 
    \begin{bmatrix}
        I & J' \\
         & J
    \end{bmatrix},
    \quad
    V_1 = 
    \begin{bmatrix}
        J' &  \\
        J  & I 
    \end{bmatrix},
    \\
    &W_0 = 
    \begin{bmatrix}
        I  \\ \\
    \end{bmatrix},
    \quad
    W_1 = 
    \begin{bmatrix}
        J' \\ J
    \end{bmatrix},
\end{split}
\]
and \[I = \begin{bmatrix}1 & 0 \\ 0 & 1 \end{bmatrix}, \quad J = \begin{bmatrix}0 & 1 \\ 0 & 0 \end{bmatrix},\quad J' = \begin{bmatrix}0 & 0 \\ 1 & 0 \end{bmatrix}.\]
\end{lemma}

Corollary~\ref{cor:simple-roots} provides an explicit formula in case of the simple roots of $g(z)$ and $h(z)$, which allows us to write $b_j$ as a weighted sum of the exponents of the form $z_t^{\pm j}$.
The next proposition provides an explicit QTT representation in this case.

\begin{proposition} \label{prop:sum_z}
    Let $\LL >2$ and consider a circulant $\invA_\LL \in \mathbb{C}^{2^\LL \times 2^\LL}$ defined by its first column
    \[
        (\invA_\LL)_j = \alpha_1 w_1^j + \dots + \alpha_r w_r^j,
    \]
    where $\alpha_t,w_t\in\mathbb{C}$, $t=1,\dots,r$, are given constants.
    Then $\invA_\LL$ admits an explicit QTT representation with the ranks $(2,r+1,r+1\dots, r+1)$:
    \[
        \invA_\LL = Q_1 \Join Q_2\Join \dots \Join Q_\LL,
    \]
    where 
    \[
    \begin{split}
        &Q_1 = 
        \begin{bmatrix}
            I & H
        \end{bmatrix},
        \\
        &Q_2 = 
        \begin{bmatrix}
            I
            &
            \KK_{1,2} & \dots & \KK_{r,2} 
            \\
             &
             \MM_{1,2} & \dots & \MM_{r,2}
        \end{bmatrix},
        \\
    &Q_k = 
    \begin{bmatrix}
        I
        &
        \KK_{1,k} & \cdots & \KK_{r,k}
        \\
        & \MM_{1,k} & \\
        & & \ddots  & \\
        & & & \MM_{r,k}
    \end{bmatrix}, 
    \quad k =3,\dots,\LL-1, \\
    &Q_\LL = 
    \begin{bmatrix}
        \left(\sum_{t} \alpha_t \right) I + \left(\sum_{t} \alpha_t w_{t} \right) J' +  \left(\sum_{t} \alpha_t w_{t}^{2^{\LL}-1} \right) J \\
        \alpha_1 w_1 \MM_{1,\LL}\\
        \vdots \\
        \alpha_{r} w_{r} \MM_{r,\LL} \\
    \end{bmatrix}
    \end{split}
    \]
    and where for all $t=1,\dots,r$, $k=1,\dots, \LL$
    \begin{align*}
        \KK_{t,k} 
        &=
        J' + q_{t,k}^{2^{k}-2} J, \\
        \MM_{t,k} &= q_{t,k}I + q_{t,k}^2 J' + J, \\
        q_{t,k} &= w_t^{2^{\LL-k}}.
    \end{align*}
\end{proposition}

\begin{proof}
	    See the proof in Appendix~\ref{app:qtt_repr}.
	\end{proof}

Despite the fact that Proposition~\ref{prop:sum_z} provides explicit QTT formulas for the circulant inverse in the case of simple roots, it is not robust in this form.
Indeed, some of the exponents will have the form $w_t^{-j}$, $j=0,\dots,2^{\LL}-1$, and since $|w_{t}|<1$, we will get $|w_t^{-2^{\LL}}| \gg 1$ for large $\LL$.
To avoid this issue, we modify Proposition~\ref{prop:sum_z} as follows. 

\begin{corollary} \label{prop:sum_z_stable}
    Let $\LL >2$ and consider a circulant $\invA_\LL \in \mathbb{C}^{2^\LL \times 2^\LL}$ defined by its first column
    \[
        (\invA_\LL)_j = \alpha_1 w_1^j + \dots + \alpha_{r_1} w_{r_1}^j + \beta_1 z_1^{2^{\LL}-j} + \dots + \beta_{r_2} z_{r_2}^{2^{\LL}-j},
    \]
    where $\alpha_t,w_t\in\mathbb{C}$, $t=1,\dots,r_1$ and $\beta_t,z_t\in\mathbb{C}$, $t=1,\dots,r_2$, are given constants.
    Then $\invA_\LL$ admits an explicit QTT representation with the ranks $(2,r_1+r_2+1,r_1 + r_2 +1\dots, r_1 + r_2 +1)$:
    \[
        \invA_\LL = Q_1 \Join Q_2\Join \dots \Join Q_\LL,
    \]
    where 
    \[
    \begin{split}
        &Q_1 = 
        \begin{bmatrix}
            I & H
        \end{bmatrix},
        \\
        &Q_2 = 
        \begin{bmatrix}
            I
            &
            \KK_{1,2} & \dots & \KK_{r_1,2} 
            &
            \KKother_{1,2} & \dots & \KKother_{r_2,2}
            \\
             &
             \MM_{1,2} & \dots & \MM_{r_1,2}
             &
             \MMother_{1,2} & \dots & \MMother_{r_2,2}
        \end{bmatrix},
        \\
    &Q_k = 
    \begin{bmatrix}
        I
        &
        \KK_{1,k} & \cdots & \KK_{r_1,k}
        &
        \KKother_{1,k} & \cdots & \KKother_{r_2,k}
        \\
         &
         \MM_{1,k} & & & & & \\
         & & \ddots  & & & \\
         & & & \MM_{r_1,k} & & & \\
         & & & & \MMother_{1,k} & & \\
         & & & & & \ddots & \\ 
         & & & & &        & \MMother_{r_2,k}
    \end{bmatrix}, 
    \quad k =3,\dots,\LL-1, \\
    &Q_\LL = 
    \begin{bmatrix}
        \gamma_1 I + \gamma_2 J' + \gamma_3 J \\
        \alpha_1 w_1 \MM_{1,\LL}\\
        \vdots \\
        \alpha_{r_1} w_{r_1} \MM_{r_1,\LL} \\
        \beta_{1} z_{1} \MMother_{1,\LL}\\
        \vdots \\
        \beta_{r_2} z_{r_2} \MMother_{r_2,\LL}
    \end{bmatrix}.
    \end{split}
    \]
    and where 
    \begin{align*}
        \KK_{t,k} 
        &=
        J' + q_{t,k}^{2^{k}-2} J,~~
        \MM_{t,k} = q_{t,k}I + q_{t,k}^2 J' + J,~~
        q_{t,k} = w_t^{2^{\LL-k}},\\
        \KKother_{t,k} 
        &=
        \qq_{t,k}^2 J' + J,~~
        \MMother_{t,k} = \qq_{t,k} I + J' + \qq_{t,k}^{2} J,~~
        \qq_{t,k} = z_t^{2^{\LL-k}},\\
        \gamma_1 
        &=
        \sum_{t=1}^{r_1} \alpha_t + \sum_{t=1}^{r_2} \beta_t z_t^{2^\LL}, \\
        \gamma_2
        &=
        \sum_{t=1}^{r_1} \alpha_t w_{t}
        +
        \sum_{t=1}^{r_2} \beta_t z_t^{2^{\LL}-1}, \\
        \gamma_3 
        &=
        \sum_{t=1}^{r_1} \alpha_t w_{t}^{2^{\LL}-1} 
        +
        \sum_{t=1}^{r_2} \beta_t z_t.
    \end{align*}
\end{corollary}
\begin{proof}
Let us denote $r \equiv r_1 + r_2$ and $\bbeta_t \equiv \beta_t z_t^{2^\LL}$, $\ww_t \equiv z_t^{-1}$.
Now we can apply Proposition~\ref{prop:sum_z} and obtain a QTT decomposition with the desired ranks:
\[
B_L = Q_1' \Join Q_2' \Join \dots \Join Q_L'.
\]
Note that the cores $Q_1', \dots, Q_{\LL-1}'$ already have the required structure, but instead of $\KKother_{k,t}$ and $\MMother_{k, t}$ we have $K'_{k,t}$ and $M'_{k,t}$:
\[
K'_{t,k} \equiv J' + z_t^{-2^\LL + 2^{\LL-k+1}}J,~~~~
M'_{t,k} \equiv z_t^{-2^{\LL-k}} I + z_t^{-2^{\LL-k+1}} J'  + J.
\]
Note the following identities: 
\begin{align*}
    \KKother_{t,k} &= K'_{t,k}z_t^{2^\LL - 2^{\LL-k+1}} = K'_{t,k}\qq_{t,k}^{2^k-2}, \\
    \MMother_{t,k} &= M'_{t,k}{z_t}^{2^{\LL-k+1}} = M'_{t,k}\frac{\qq_{t,k}^{2^k-2}}{\qq_{t,k-1}^{2^{k-1}-2}}.
\end{align*}
Thus,
\[
Q'_{k}
\Join
\begin{bmatrix}
    1 &        &   &                         &        & \\
      & \ddots &   &                         &        & \\
      &        & 1 &                         &        & \\
      &        &   & \qq_{1,k}^{2^k-2}               &        & \\
      &        &   &                         & \ddots & \\
      &        &   &                         &        & \qq_{r_2,k}^{2^k-2} \\
\end{bmatrix}
=
\begin{bmatrix}
    1 &        &   &                         &        & \\
      & \ddots &   &                         &        & \\
      &        & 1 &                         &        & \\
      &        &   & \qq_{1,k-1}^{2^{k-1}-2}             &        & \\
      &        &   &                         & \ddots & \\
      &        &   &                         &        & \qq_{r_2,k-1}^{2^{k-1}-2} \\
\end{bmatrix}
\Join
Q_k.
\]
Propagating the above diagonal matrix from right to left, we can prove that
\[
Q'_1\Join \dots \Join Q'_{\LL-1}
\Join
\begin{bmatrix}
    1 &        &   &                         &        & \\
      & \ddots &   &                         &        & \\
      &        & 1 &                         &        & \\
      &        &   & \qq_{1,\LL-1}^{2^{\LL-1}-2}             &        & \\
      &        &   &                         & \ddots & \\
      &        &   &                         &        & \qq_{r_2,\LL-1}^{2^{\LL-1}-2}
\end{bmatrix}
=
Q_1\Join \dots \Join Q_{\LL-1}.
\]

Now consider $Q_{\LL}'$.
Its elements differing from those of $Q_L$ are 
$\bbeta z_t^{-1} M'_{t,\LL} = z_t^{2^\LL-4} \beta_t z_t \MMother_{t,\LL}$.
It remains to notice that $\qq_{1,\LL-1}^{2^{\LL-1}-2} = z_t^{2^\LL-4}$, so
\[
Q_{\LL}' = 
\begin{bmatrix}
    1 &        &   &                         &        & \\
      & \ddots &   &                         &        & \\
      &        & 1 &                         &        & \\
      &        &   & \qq_{1,\LL-1}^{2^{\LL-1}-2}             &        & \\
      &        &   &                         & \ddots & \\
      &        &   &                         &        & \qq_{r_2,\LL-1}^{2^{\LL-1}-2}
\end{bmatrix}
\Join
Q_{\LL}.
\]
Thus, $B_L = Q_1'\Join \dots \Join Q_{\LL}' = Q_1 \Join \dots \Join Q_{\LL}$.
\end{proof}

\section{Inversion of one-dimensional stiffness and mass matrices} \label{sec:examples}

In this section, we consider several well-known examples of matrices, arising from the discretization of second order one-dimensional  periodic boundary value problem with constant coefficients
on uniform grids and piecewise linear finite elements.
Namely in this section, we consider the inversion of a mass matrix $\mass\in \mathbb{R}^{N \times N}$ and a stiffness matrix~$\lap\in \mathbb{R}^{N \times N}$, shifted by $\shift\geq 0$:
\[
   \mass = \mathsf{circ}(4,1,0,\dots,0,1), \quad\lap + \shift I = \mathsf{circ}(2+\shift,-1,0,\dots,0,-1).
\]
Note that $\lap$ is singular, so we consider its pseudoinverse separately in Section~\ref{sec:pseudo}.

\subsection{Inversion of the mass matrix $\mass$}

    To apply Theorem~\ref{thm:general-inverse}, we first write down the polynomials $g(z)$ and $h(z)$.
    Obviously, \[g(z) = 1\cdot z^0 + 4 \cdot z^1 + 1 \cdot z^2.\]
    Note that due to the symmetry of matrix $\mass$, we have $g(z) = h(z)$.
    The roots of $g(z)$ are 
    \[z_{1} = -2+\sqrt{3} \quad \text{and}\quad z_{2} = -2- \sqrt{3}.\]
    The root $z_1$  lies inside the unit circle, $z_2$ lies outside, so according to Corollary~\ref{cor:simple-roots}, we get
    \begin{align*}
    (\mass^{-1})_{i,0} 
    &=
    \frac{1}{(z_1-z_2)(1-z_1^N)}z_1^{-i+N}
    +
    \frac{1}{(z_1-z_2)(1-z_1^N)}z_1^{i}
    = \\ &=
    \frac{1}{2\sqrt{3}(1-(\sqrt{3}-2)^N)}
    \left(
    (\sqrt{3}-2)^{N-i} + (\sqrt{3}-2)^{i}
    \right).
    \end{align*}
    Thanks to Corollary~\ref{cor:qtt-ranks-special-circulant}, the QTT ranks of $\mass^{-1}$ do not exceed $3$ (if $N = \base^\LL$) and we can directly apply Corollary~\ref{prop:sum_z_stable} to obtain the explicit QTT representation of $\mass^{-1}$.

\subsection{Inversion of the shifted stiffness matrix $\lap+  \shift I$, $\shift>0$} \label{sec:circ_1dode}
    
    Now consider the discretization of a shifted periodic Laplacian operator:
    \[
    \lap + \shift I = \mathsf{circ}(2+\shift,-1,0,\dots,0,-1) \in \mathbb{R}^{N \times N},~~\shift > 0.
    \]
    This circulant is also symmetric, so $g(z) = h(z) = -\shift^2 + (2+\shift)z - 1$.
    The roots are:
    \[
    z_{1} = 1 + \frac{\shift}{2} - \sqrt{\frac{\shift^2}{4}+\shift},
    ~~~
    z_{2} = 1 + \frac{\shift}{2} + \sqrt{\frac{\shift^2}{4}+\shift}.
    \]
    Again, $z_1$ lies inside $U$ and $z_2$ lies outside of it (this holds for any $\shift > 0$: obviously, $z_2 > 1$, and the product $z_1z_2$ must be equal to $1$ by Vieta's formulas).
    \begin{align*}
    \left((\lap + \shift I)^{-1}\right)_{i,0} 
    &=
    \frac{1}{-(z_1-z_2)(1-z_1^N)}z_1^{-i+N}
    +
    \frac{1}{-(z_1-z_2)(1-z_1^N)}z_1^{i}
    = \\ &=
    \frac{1}{\sqrt{\shift^2 + 4\shift}(1-z_1^N)}
    \left(
    z_1^{N-i} + z_1^{i}
    \right).
    \end{align*}
    Due to Corollary~\ref{cor:qtt-ranks-special-circulant}, the QTT ranks of $(\lap+\shift I)^{-1}$ do not exceed $3$ (if $N = \base^\LL$) and we can directly apply Corollary~\ref{prop:sum_z_stable} to obtain the explicit QTT representation of $(\lap + \shift I)^{-1}$.

\subsection{Pseudoinversion of the stiffness matrix $\lap$}\label{sec:pseudo}

	In this section, we discuss the rank bounds of the explicit pseudoinverse of $\lap^+$.
	The explicit formula for the pseudoinverse is by no means new, and is available, e.g., in~\cite{plonka2016pseudo}.
	Nevertheless, we still provide the derivation to illustrate that the proposed approach of finding pseudoinverses can be automated with the help of \texttt{Sympy} Python package~\cite{sympy} or Wolfram \texttt{Mathematica}~\cite{Mathematica}.
	
	We start with the well-known formula~\cite{golub2013matrix}:
	\begin{equation}\label{eq:pseudoinverse-lim}
	    A^+ = \lim_{\alpha \to 0+} (A^*A + \alpha I)^{-1}A^*.
	\end{equation}
	If the circulant $A^*A + \alpha I$ satisfies the conditions of Theorem~\ref{thm:general-inverse}, we can find an explicit formula the elements of $(A^*A + \alpha I)^{-1}$ and then for $(A^*A + \alpha I)^{-1}A^*$.
	Computing the limit for $\alpha \to 0+$ is a tedious but solely technical task, and the aforementioned symbolic algebra libraries can facilitate it. 
	In particular, we have used \texttt{Sympy} for this purpose.
	
	To elaborate on the proposed idea, we need to understand the form of Laurent polynomial of the product of circulants $A$ and $B$ with known polynomials $f_A(z)$ and $f_B(z)$. The proof of this proposition is technical and straightforward, but for completeness we provide it in Appendix~\ref{app:examples} since we have not been able to find the proof in this specific setting of the proposition.
	
	\begin{proposition}\label{prop:laur_prod}
	Let $A,B \in \mathbb{C}^{N \times N}$ be circulants of the form~\eqref{eq:A} (with parameters $m_A, n_A$ and $m_B, n_B$ respectively).
	Moreover, let's assume that $N \ge m_A+n_A+m_B+n_B$.
	Then the circulant $C \equiv AB$ also has form~\eqref{eq:A} with parameters $m = m_A + m_B$ and $n = n_A+n_B$ and its corresponding Laurent polynomial of the form~\eqref{eq:laurent} is $f_C(z) = f_A(z)f_B(z)$, where $f_A(z)$ and $f_B(z)$ are Laurent polynomials corresponding to $A$ and $B$.
	\end{proposition}
	\begin{proof}
	    See the proof in Appendix~\ref{app:examples}.
	\end{proof}
	
	\begin{corollary}\label{cor:AstarA}
	Let $A \in \mathbb{R}^{N \times N}$ be a real symmetric circulant of the form~\eqref{eq:A} with corresponding Laurent polynomial $f_A(z)$.
	If $N \ge 2(m+n)$, then $C \equiv A^*A+\alpha I$ is a real symmetric circulant of the form~\eqref{eq:A} with $m = 2m_A$, $n = 2n_B$ and Laurent polynomial $f_C(z) = (f(z))^2 + \alpha$.
	\end{corollary}
	
	Now we can demonstrate the steps of the proposed method on $\lap$.
	The corresponding Laurent polynomial is $f_S(z) = 2 - z-z^{-1}$.
	According to Corollary~\ref{cor:AstarA}, the polynomial corresponding to $\stiffnessQ_{\alpha} \equiv A_S^*A_S + \alpha^4 I$ is $f_\stiffnessQ(z) = (2-z-z^{-1})^2 + \alpha^4$
	(obviously, we can use $\alpha^4$ in the equation~\eqref{eq:pseudoinverse-lim} instead of $\alpha$).
	To apply Theorem~\ref{thm:general-inverse} we need to solve the equation $f_\stiffnessQ(z) = 0$ (it is equivalent to $f_\stiffnessQ(z)z^2 = 0$).
	It reduces to two equations $f_S(z) = \pm i\alpha^2$.
	The resulting roots are 
	\[
	z_{1,2} = 1 - \frac{\alpha}{2}\sqrt{\pm 4i - \alpha^2} \pm \frac{i\alpha^2}{2},~~~z_{3,4} = 1 + \frac{\alpha}{2}\sqrt{\pm 4i -\alpha^2} \pm \frac{i\alpha^2}{2}.
	\]
	It is not difficult to check that for all sufficiently small $\alpha > 0$ the roots $z_1$ and $z_2$ lie inside the unit circle $U$, whereas $z_3$ and $z_4$ lie outside of it.
	Moreover, it is obvious that all four roots are distinct for any $\alpha > 0$, so to compute the elements of $\stiffnessQ_{\alpha}^{-1}$ we can apply Corollary~\ref{cor:simple-roots}:
	\begin{align*}
	(\stiffnessQ_\alpha^{-1})_{i,0} &= 
		\frac{z_{1}^{N - i + 1}}{\left(1 - z_{1}^{N}\right) \left(z_{1} - z_{2}\right) \left(z_{1} - z_{3}\right) \left(z_{1} - z_{4}\right)}
	+
	\frac{z_{2}^{N - i + 1}}{\left(1 - z_{2}^{N}\right) \left(- z_{1} + z_{2}\right) \left(z_{2} - z_{3}\right) \left(z_{2} - z_{4}\right)}
	+\\&+
	\frac{z_{1}^{i + 1}}{\left(1 - z_{1}^{N}\right) \left(z_{1} - z_{2}\right) \left(z_{1} - z_{3}\right) \left(z_{1} - z_{4}\right)}
	+
	\frac{z_{2}^{i + 1}}{\left(1 - z_{2}^{N}\right) \left(- z_{1} + z_{2}\right) \left(z_{2} - z_{3}\right) \left(z_{2} - z_{4}\right)}.
    \end{align*}
    Reducing to a common denominator, we come to 
    \[
    (\stiffnessQ_\alpha^{-1})_{i,0}
    =
    \frac
    {\left(z_{1}^{N} - 1\right)\left(z_{1} - z_{3}\right) \left(z_{1} - z_{4}\right)  \left(z_{2}^{i + 1} + z_{2}^{N - i + 1}\right)
    -
    \left(z_{2}^{N} - 1\right)\left(z_{1}^{i + 1} + z_{1}^{N - i + 1}\right) \left(z_{2} - z_{3}\right) \left(z_{2} - z_{4}\right) }
    {\left(z_{1}^{N} - 1\right) \left(z_{2}^{N} - 1\right) \left(z_{1} - z_{2}\right) \left(z_{1} - z_{3}\right) \left(z_{1} - z_{4}\right) \left(z_{2} - z_{3}\right) \left(z_{2} - z_{4}\right) }.
    \]
    The first column of $\stiffnessQ_\alpha^{-1}S$ is simply
    \[
    (\stiffnessQ_\alpha^{-1}S)_{:,0} = 2(\stiffnessQ_\alpha^{-1})_{:,0} - (\stiffnessQ_\alpha^{-1})_{:,1} - (\stiffnessQ_\alpha^{-1})_{i,N-1},
    \]
    and the explicit expression for it can be written down.
    After this we have used the Sympy library to compute the Taylor series of both the numerator and denominator.
    It turned out that the numerator is 
    \[
    \frac{N \left(96 \sqrt{2} i^{2} - 96 \sqrt{2} N i + 16 \sqrt{2} N^{2}  - 16 \sqrt{2}\right)\sqrt{-1}}{24}\alpha^7 + O(\alpha^8)
    \]
    and the denominator is $(8\sqrt{2}N^2\sqrt{-1})\alpha^7 + O(\alpha^8)$.
    Here $\sqrt{-1}$ denotes the imaginary unit.
    Dividing and taking limit of $(\stiffnessQ_\alpha^{-1}S)_{i,0}$ for $\alpha \to 0$ we conclude that
    $(S^+)_{i,0} = (6i^2 - 6Ni + N^2-1)/(12N)$. 
    Thus, we arrive to the same expression as in~\cite{plonka2016pseudo}.
	\begin{proposition}
		The pseudoinverse of $\lap$ is a circulant $\lap^+$ with elements $(\lap^+)_{i,j} = f((i-j)\bmod N)$ where 
		\[
			f(i) = \frac{6i^2 - 6Ni + N^2-1}{12N}.
		\]
	\end{proposition}
	\begin{corollary}
	For any positive integer $\LL$ the QTT ranks of pseudoinverse $\lap^+$ of stiffness matrix $\lap \in \mathbb{R}^{\base^\LL \times \base^\LL}$ do not exceed $4$.
	\end{corollary}

\section{Numerical experiments} \label{sec:num_exp}

\subsection{One-dimensional convection-reaction-diffusion equation}

The goal of these numerical experiments is to justify the robustness of the derived explicit QTT formulas for large values of $\LL$. 
As an example, we consider the one-dimensional convection-reaction-diffusion boundary value problem with periodic boundary conditions:
\begin{equation}\label{eq:pde1d}
\begin{aligned}
    -u''(x) + u'(x) + u(x) &= f(x),\quad x\in (0,1) \\
    u(0) &= u(1).
\end{aligned}
\end{equation}
In particular, we set $u(x) = \cos (2\pi x)$ and obtain the right-hand side: \[f(x) = (4\pi^2+1) \cos (2\pi x) - 2\pi \sin (2\pi x),\]
which we use to recover the $u(x)$.

The finite-difference discretization of~\eqref{eq:pde1d} on a uniform grid with the grid step size $h \equiv 2^{-\LL}$ and forward differences applied to the convection term $u'$ leads to the following system of $2^L$ linear equations: 
$
A_h u_h = h^2 f_h,
$
where
\[
A_h \equiv \mathsf{circ}(2-h+h^2, -1, 0, \dots, -1+h)
\]
is a non-symmetric circulant matrix.
The right-hand side $f_h$ is assembled in the QTT format using the cross approximation method~\cite{ot-ttcross-2010}.
If we obtain a QTT decomposition for the matrix $A_h^{-1}$, the solution $u_h = A_h^{-1} f_h$ may be found efficiently through QTT matrix-vector product, which admits explicit representation in terms of the QTT cores of both $A_h^{-1}$ and $f_h$~\cite{osel-tt-2011}.

Thanks to Theorem~\ref{thm:general-inverse}, we know that the first column of $A_h^{-1}$ has elements $b_j = c_1z_1^{N-j} + c_2w_1^{j}$ where $c_1, c_2, z_1$ and $w_1$ can be found analytically.
Next, we apply Proposition~\ref{prop:sum_z_stable} to construct the explicit QTT decomposition of $A_h^{-1}$ with the ranks $(2,3,\dots,3)$.

The comparison of the black-box optimization-based TT solver AMEn (alternating minimal energy method)~\cite{ds-amen-2014} with the proposed approach is shown in Figure~\ref{fig:numerical}.
As expected, the proposed approach appears to be stable for a wide range of $\LL$, while the AMEn solver, applied to $A_h u_h = f_h$, becomes unstable for $\LL\gtrsim 20$.
We note that the instabilities arising for AMEn are not related to the solver itself, but rather to the ill conditioning of $A_h$~.

\begin{figure}
    \centering
    \includegraphics[width=0.65\textwidth]{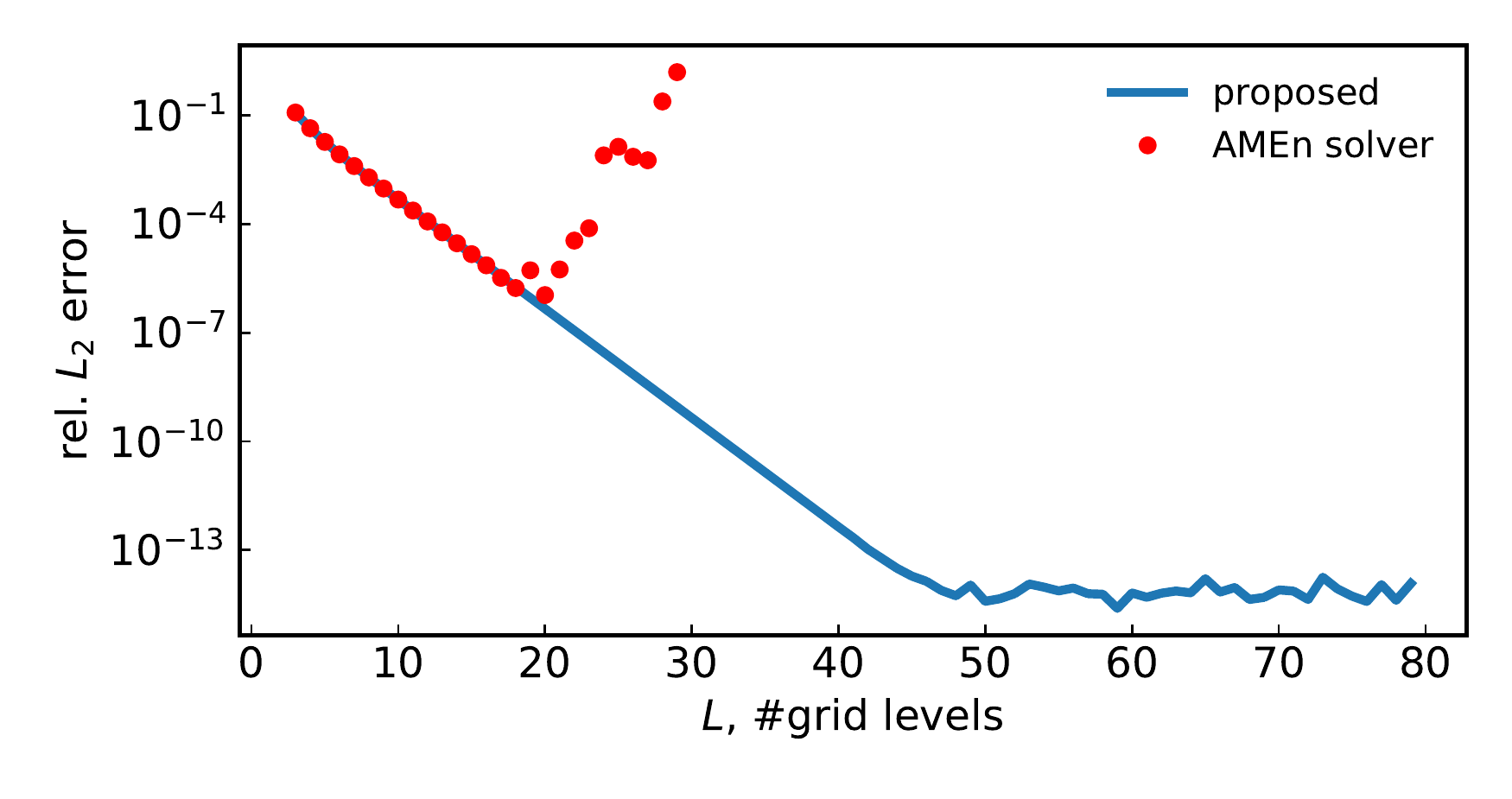}
    \caption{Relative $\mathrm{L}_2$-errors against the number of grid levels $\LL$ (the total number of grid points is $2^\LL$) for the solutions obtained by solving $A_h u_h = f_h$ using the AMEn solver and by directly computing $u_h = A_h^{-1}f_h$ as a QTT matrix-by-vector product using the proposed formulas for $A_h^{-1}$.}
    \label{fig:numerical}
\end{figure}

\begin{remark}
To construct the cores of the QTT decomposition of $A_h^{-1}$, we need to compute the numbers of the form $z^M$ for $z = 1 - \gamma_1 h + \gamma_2 h^2 + O(h^3)$.
For large values of $M$ (e.g. $M = 2^\LL$) and small values of $h$ (e.g., $h \approx \sqrt{\varepsilon_{\mathrm{machine}}}$), the direct computation of $z^M$ gives rise to the error of the order $\sqrt{\varepsilon_{\mathrm{machine}}}$.
It happens as the $\gamma_2h^2$ term is ``lost'' during the computation of $z$, whereas the following shows that it has $\mathcal{O}(h)$ impact on $z^{2^L}$:
\begin{align*}
z^{2^L} &= z^{\frac{1}{h}} = \exp\left(\frac{1}{h}\ln(1 - \gamma_1 h + \gamma_2 h^2 + O(h^3))\right) = \\ &= \exp\left(\frac{1}{h}(-\gamma_1 h + (\gamma_2^2-\gamma_1^2)h^2 + O(h^3))\right)
= \exp(-\gamma_1 + (\gamma_2^2 - \gamma_1^2)h + O(h^2)).
\end{align*}
To retain accuracy of order $\varepsilon_{\mathrm{machine}}$, we have used the above expansion instead of the naive computation of~$z^M$.
\end{remark}

\subsection{Three-dimensional screened Poisson equation}

For $x = (x_1,x_2,x_3) \in \mathbb{R}^3$, let us denote $|x|= \sqrt{x_1^2+x_2^2 + x_3^3}$. 
We consider a three-dimensional screened Poisson equation:
\begin{equation}\label{eq:screened3d}
    -\Delta u + u = 2|x|^{-1} e^{-|x|}, \quad x  \in \Omega = (-a,a)^3,
\end{equation}
with periodic boundary conditions on all opposite faces of the cube $\Omega$. 
It can be straightforwardly verified that $u(x)= e^{-|x|}$ satisfies~\eqref{eq:screened3d}.
To ensure that it also satisfies boundary conditions with high precision, we select $a = 40$, which implies that both values and gradients of $u(x)$ on $\partial \Omega$  are zeroes up to machine epsilon: $e^{-40} \approx 4.2 \cdot 10^{-18}$. 

We discretize the equation using finite difference method on a uniform $2^L \times 2^L \times 2^L$ grid, which leads to a linear system with a matrix of the form
\begin{equation}\label{eq:screened3d_mat}
    A = A_S \otimes I \otimes I + I \otimes A_S \otimes I + I \otimes I \otimes A_S + h^2 I \otimes I \otimes I, \quad A_S = \mathsf{circ}(2, -1, 0, \dots, -1),
\end{equation}
where $h = 2a/2^{L}$. The right-hand side is assembled using exponential sums as described in~\cite{rakhuba2021robust}. To robustly solve equation with the matrix~\eqref{eq:screened3d_mat}, we utilize the idea from~\cite{rakhuba2021robust} and apply the tensor version of the alternating direction implicit (ADI) method. 
This method is based on explicit inversions of shifted discretized one-dimensional operators in the QTT format.
In our case, to run the ADI method, we need to have access to the explicit QTT representation of matrices of the form $(A_S + \shift I)^{-1}$, $\shift>0$, which we have already derived in Section~\ref{sec:circ_1dode}.
The explicit inversions are then used to construct the ADI transition operator~\cite{rakhuba2021robust} in the combined Tucker and QTT (TQTT) format, which is more efficient for three-dimensional problems compared with the original QTT format.
The code for the TQTT-ADI method with the proposed explicit formulas is available at~\url{https://bitbucket.org/rakhuba/qttcirc}.

\begin{figure}
    \centering
    \includegraphics[width=0.49\textwidth]{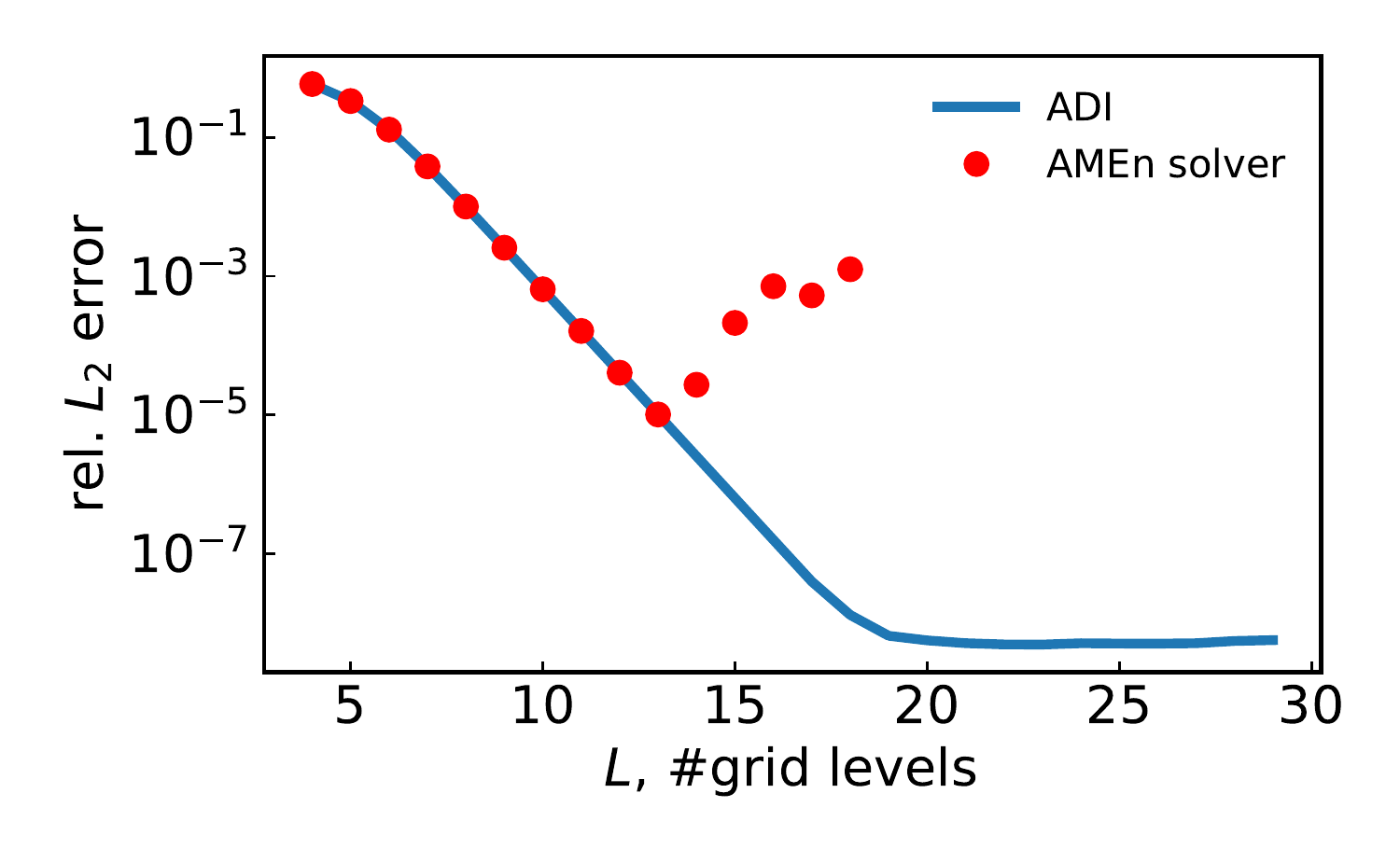}
    \hfill
    \includegraphics[width=0.49\textwidth]{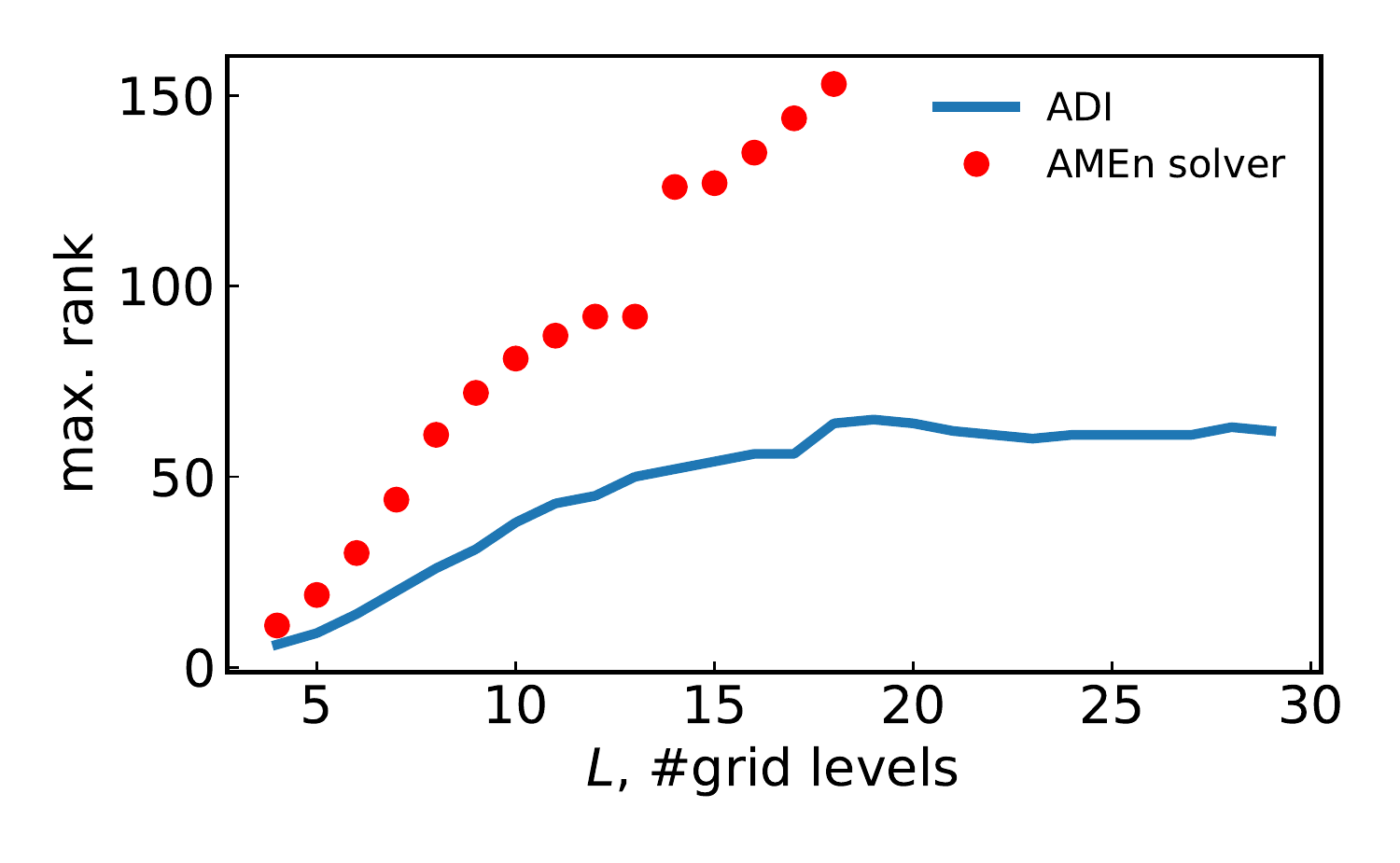}
    \caption{Relative $\mathrm{L}_2$-errors (left) and maximum ranks of the solution (right) against the number of grid levels $\LL$ (the total number of grid points is $2^{3\LL}$) to solve the discretized equation~\eqref{eq:screened3d} with periodic boundary conditions. ADI stands for the solver~\cite{rakhuba2021robust} in the TQTT format, combined with the proposed circulant inversion formulas. The AMEn solver is performed in the QTT format. For both solvers the rank truncation parameter $\varepsilon = 10^{-9}$ is utilized.}
    \label{fig:3d_errs}
\end{figure}
In Figure~\ref{fig:3d_errs}, we present $\mathrm{L}_2$-errors with respect to $e^{-|x|}$ and maximum ranks for the TQTT-ADI method (combined with the proposed circulant inversion formulas) and the AMEn solver, applied to~\eqref{eq:screened3d_mat} in the QTT format.
Similarly to the one-dimensional case, the errors from the AMEn solver start increasing after a certain number of grid levels.
At the same time, the proposed approach is capable of maintaining the desired accuracy level for large $L$.
The plot with maximum rank values also shows that the ranks using the ADI method stabilize, while for the AMEn solver they start increasing in the region of instabilities.
We also note that the AMEn solver is available only for the QTT format. 
This explains why the rank values are larger even in the region with no stability issues.
The fact that QTT ranks are larger than those in the TQTT format was also observed in~\cite{marcati2019tensor,rakhuba2021robust}.

In Figure~\ref{fig:3d_times}, we present computation times of the TQTT-ADI method for different rank truncation parameters $\varepsilon$.
The figure suggests that for $L=25$ we are able to solve the system within a minute of computation time for all considered $\varepsilon$.
At the same time, in the given range of $L$ it was not even possible to run methods that require storing full tensors: for $L=10, 25$ storing a single tensor of the size $2^{L}\times 2^L \times 2^L$ would require $\sim8$ Gb and $3\cdot 10^{14}$ Gb respectively.

\begin{figure}
    \centering
    \includegraphics[width=0.65\textwidth]{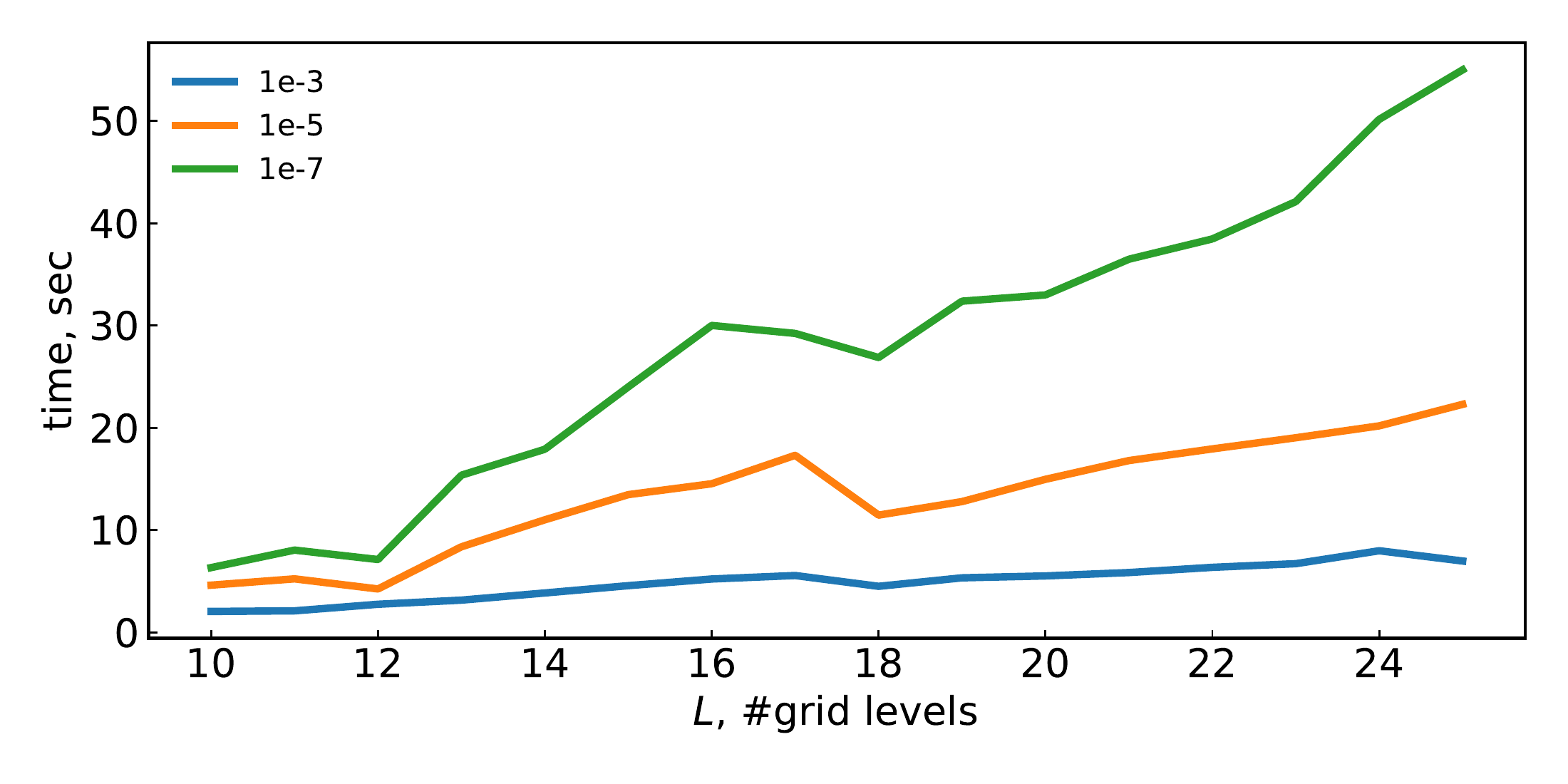}
    \caption{Computation times of the ADI method in the TQTT format with the proposed circulant inversion formulas against the number of grid levels $\LL$ (the total number of grid points is $2^{3\LL}$) to solve the discretized equation~\eqref{eq:screened3d} with periodic boundary conditions.
    Different lines correspond to different  rank truncation parameters $\varepsilon$.}
    \label{fig:3d_times}
\end{figure}

\section*{Acknowledgments}
This work is supported by Russian Science Foundation grant \textnumero\;21-71-00119.

\bibliographystyle{plain}

\appendix

\section{Proofs of Section~\ref{sec:circ}} \label{app:circ}

\begin{lemma}\label{lm:sum-shift}
	For any function $f(j):\{0,\dots,N-1\}\to \mathbb{C}$ and integers  $s,t$ the following holds:
	\[
	\sum_{j=s}^{s+N-1} f(j \bmod N) = \sum_{j=t}^{t+N-1} f(j \bmod N).
	\]
	\end{lemma}
\begin{proof}[Proof of Lemma~\ref{lm:sum-shift}]
For any integer $s$ the sequence 
	\[
	\Big(s \bmod N, \dots, (s+N-1) \bmod N\Big)
	\]
	is obviously a permutation of $(0,\dots, N-1)$, thus
	\[
	\sum_{j=s}^{s+N-1} f(j \bmod N) = \sum_{j=0}^{N-1} f(j) = \sum_{j=t}^{t+N-1} f(j \bmod N).
	\]
	\end{proof}

	\begin{proof}[Proof of Lemma~\ref{lm:equivalent}]
	We start from~\eqref{eq:inverse} and substitute the summation index with $j' = j - \ell$:
	\[
    \sum_{j'=-\ell}^{N-1-\ell} A_{((k-\ell)- j') \bmod N, 0}~b_{j'\bmod N}
	=
	\delta_{k,\ell},\quad k,\ell \in \{0, \dots, N-1\}.
	\]
	We can apply Lemma~\ref{lm:sum-shift} to the left part of the equality as the summed expression is indeed a function of $j' \bmod N$.
    Therefore, we obtain:
	\[
	\sum_{j'=0}^{N-1} A_{((k-\ell)- j') \bmod N, 0}~b_{j'}
	=
	\delta_{k,\ell},\quad k,\ell \in \{0, \dots, N-1\}.
	\]
	The left part of the equality depends only on $(k-\ell) \bmod N$ and thus instead of $N^2$ equations we can equivalently write only $N$:
	\[
	\sum_{j=0}^{N-1} A_{(k-j) \bmod N, 0}~b_j
	=
	\delta_{k,0},~ 0\le k \le N-1,\quad k,\ell \in \{0, \dots, N-1\}.
	\]
	If we substitute the summation index with $j'' = k-j$ and take into account that $x_j = \xi_{j}$ for $0\le j \le N-1$, we obtain
	\[
	\sum_{j''=k-m-n+1}^{k} A_{j'' \bmod N, 0}~\xi_{k-j''} = \delta_{k,0}.
	\]
	We again apply Lemma~\ref{lm:sum-shift} (the biinfinite vector $\xi$  is $N$-periodic and thus depends only on $j'' \bmod N$):
	\[
	\sum_{j''=-n}^{m-1}A_{j'' \bmod N, 0}~\xi_{k-j''} = \delta_{k,0}.
	\]
	It is obvious that for $-n \le j'' \le m-1$ it holds that $A_{j'' \bmod N,0} = \Ainf_{j'', 0}$.
	Moreover, $\Ainf_{j'', 0} = 0$ for $j'' \not\in [-n,m-1]$, so we obtain
	\[
	\sum_{j''=-\infty}^{\infty}\Ainf_{j'',0}~\xi_{k-j''} = \delta_{k,0}.
	\]
	Change the summation index back to $j = k-j''$:
	\[
	\sum_{j=-\infty}^{\infty}\Ainf_{k-j,0}~\xi_{j}
	=
	\sum_{j=-\infty}^{\infty}\Ainf_{k,j}~\xi_{j}
	=
	\delta_{k,0},
	,~ 0\le k \le N-1.
	\]
	Now take any $k' \in \mathbb{Z}$ and represent it as $k' = k + Nq$, where $q$ and $k$ are integers such that $0 \le k \le N-1$.
	By changing the summation index to $j' = j - Nq$ we can write:
	\[
	\sum_{j=-\infty}^{\infty}\Ainf_{k',j}\xi_{j}
    =
    \sum_{j'=-\infty}^{\infty}\Ainf_{k',j'+Nq}\xi_{j'}
    =
    \sum_{j'=-\infty}^{\infty}\Ainf_{k,j'}\xi_{j'}
    =
    \delta_{k,0}.
    \]
    Therefore, the system of equations~\eqref{eq:inverse} is equivalent to the infinite system of equations
    \[
    \sum_{j=-\infty}^{\infty}\Ainf_{k,j}\xi_{j} = \delta_{k\bmod N,0},~k\in\mathbb{Z},
    \]
    which is the same as $\Ainf\xi = \beta$.
	\end{proof}

	\begin{proof}[Proof of Lemma~\ref{lm:unnamed}]
	Let us compute an element of matrix $\Ainf \Binf$:
	\begin{align*}
	\sum_{j=-\infty}^{\infty}\Ainf_{k, j}\Binf_{j, \ell}
	&=
	\sum_{j=-\infty}^{\infty}\Ainf_{k-j, 0}\Binf_{j, \ell}
	=
	\sum_{j'=-n}^{m-1}a_{j'}\Binf_{k-j', \ell}
	=\\&=
	\frac{1}{2\pi i}\oint_U\frac{1}{g(z)}\sum_{j'=-n}^{m-1}a_{j'}z^{l-k+j'-1}dz
	=
	\frac{1}{2\pi i}\oint_U\frac{g(z)}{g(z)}z^{l-k-1}dz.
	\end{align*}
	But for any integer $q$ it holds:
	\[
	\oint_U z^{q}dz
	=
	\begin{cases}
	2\pi i,&\text{if } q = -1,\\
	0,&\text{otherwise}.
	\end{cases}
	\]
	Therefore, we finally obtain
	\[
	\sum_{j=-\infty}^{\infty}\Ainf_{k, j}\Binf_{j, \ell} = \delta_{k,\ell} ~\Longleftrightarrow~\Ainf \Binf = I^{(\infty)}.
	\]
	\end{proof}
	
\section{Proofs of Section~\ref{sec:qtt_repr}} \label{app:qtt_repr}

\begin{proof}[Proof of Proposition~\ref{prop:sum_z}]
First, we represent the circulant using the powers of permutation matrix $\perm$:
\[
    \invA_\LL = \sum_{i_1,\dots,i_\LL=0}^{1,\dots,1} \sum_{t=1}^r \alpha_t z_t^{\left(2^{\LL-1}i_{\LL} + \dots + 2^1 i_2 + i_1\right)} \perm_\LL^{\,\overline{i_1\dots, i_\LL}}.
\]
Then use the result of Lemma~\ref{lm:perm} for $\perm_\LL^{\,\overline{i_1\dots, i_\LL}}$ and the polylinearity of the TT decomposition:
\[
\begin{split}
    \invA_\LL 
    &= \sum_{i_1,\dots,i_\LL=0}^{1,\dots,1}\ \sum_{t=1}^r
    \alpha_t
    z_t^{\left(2^{\LL-1}i_{\LL} + \dots + 2^1 i_2 + i_1\right)} \ U_{i_\LL} \Join V_{i_{\LL-1}} \Join \dots \Join V_{i_{2}} \Join W_{i_1}  
    = \\
    &=
    \sum_{t=1}^r \left( \sum_{i_\LL=0}^1 z_t^{2^{\LL-1}i_{\LL}} U_{i_\LL}  \right)
    \Join 
    \dots
    \Join 
    \left( \alpha_t \sum_{i_1=0}^1 z_t^{i_{1}} W_{i_1} \right) = \\
    &=
    \sum_{t=1}^r \left(U_0 + z_t^{2^{\LL-1}} U_1 \right) \Join \left(V_0 + z_t^{2^{\LL-2}} V_1 \right) \Join \dots \Join \left(V_0 + z_t^{2^{1}} V_1 \right) \Join \left(\alpha_t \left(W_0 + z_t W_1 \right) \right) \\
    &=
    {\widetilde Q}_1 \Join {\widetilde Q}_2 \Join \dots \Join {\widetilde Q}_d,
\end{split}
\]
where thanks to Lemma~\ref{lm:sumtt}, the cores ${\widetilde Q}_k$ are block matrices of the block size $2r \times 2r$ for $k=2,\dots,\LL-1$, $1\times 2r$ for $k=1$ and $2r\times 1$ for $k=\LL$, i.e., it is an explicit QTT representation with all ranks equal to $2r$.
In the next steps, our goal is to find linear dependencies and reduce the rank values to $(2, r+1, \dots, r+1)$.
For the ease of presentation, we provide the proof for $r=2$. The generalization to $r>2$ is straightforward.
We have
\[
\begin{split}
    {\widetilde Q}_1
    &= 
    \begin{bmatrix}
        I + q_{1,1} H & H + q_{1,1} I & I + q_{2,1} H & H + q_{2,1} I
    \end{bmatrix}
    \\
    &= 
    \begin{bmatrix}
        I & H
    \end{bmatrix}
    \Join
    \begin{bmatrix}
        1 & q_{1,1} & 1 & q_{2,1} \\
        q_{1,1} & 1 & q_{2,1} & 1
    \end{bmatrix}
\end{split}
\]
Next, by denoting $Q_1 = \begin{bmatrix} I & H \end{bmatrix}$, we have
\[
    {\widetilde Q}_1 \Join {\widetilde Q}_2 = \left(Q_1 \Join  
    \begin{bmatrix}
        1 & q_{1,1} & 1 & q_{2,1} \\
        q_{1,1} & 1 & q_{2,1} & 1
    \end{bmatrix}  \right)
    \Join
    {\widetilde Q}_2
    =
    Q_1 \Join  
    \left(
    \begin{bmatrix}
        1 & q_{1,1} & 1 & q_{2,1} \\
        q_{1,1} & 1 & q_{2,1} & 1
    \end{bmatrix}
    \Join
    {\widetilde Q}_2
    \right),
\]
where
\begin{equation}\label{eq:proof:Q}
    {\widetilde Q}_k = 
    \begin{bmatrix}
        I + q_{1,k} J' & J' \\
        q_{1,k} J & J + q_{1,k} I \\
        & & I + q_{2,k} J' & J' \\
        & & q_{2,k} J & J + q_{2,k} I
    \end{bmatrix}, 
    \quad  
    k = 2,\dots,L-1.
\end{equation}
Using the fact that $q_{t,1} = q_{t,2}^2$ for all $t$, we get:
\[
\begin{split}
    &\begin{bmatrix}
        1 & q_{1,1} & 1 & q_{2,1} \\
        q_{1,1} & 1 & q_{2,1} & 1
    \end{bmatrix}
    \Join
    {\widetilde Q}_2
     \\
    &=\begin{bmatrix}
        I + q_{1,2} J' + q_{1,2}^3 J & J' + q_{1,2}^2J + q_{1,2}^3 I & I + q_{2,2} J' + q_{2,2}^3 J & J' + q_{2,2}^2J + q_{2,2}^3 I \\
        q_{1,2}^2 I + q_{1,2}^3 J' + q_{1,2} J & q_{1,2}^2J'+ J + q_{1,2} I & q_{2,2}^2 I + q_{2,2}^3 J' + q_{2,2} J & q_{2,2}^2J'+ J + q_{2,2} I
    \end{bmatrix} 
    \\
    &= 
    \begin{bmatrix}
        I & J' + q_{1,2}^2 J &  J' + q_{2,2}^2 J \\
        O & q_{1,2}I + q_{1,2}^2 J' + J & q_{2,2}I + q_{2,2}^2 J' + J
    \end{bmatrix}
    \Join 
    \begin{bmatrix}
        1 & q_{1,2}^3 & 1 & q_{2,2}^3\\
        q_{1,2}& 1 & 0 & 0 \\
        0 & 0 & q_{2,2} & 1 
    \end{bmatrix}.
\end{split} 
\]
In the last equation we obtained a factorization into a product of $2\times 3$ block matrix, which we denote as $Q_2$, times a $3\times 4$ matrix, which we propagate further to ${\widetilde Q}_3$ (see \eqref{eq:proof:Q} for ${\widetilde Q}_3$):
\[
\begin{split}
    &\begin{bmatrix}
        1 & q_{1,2}^3 & 1 & q_{2,2}^3 \\
        q_{1,2}& 1 & 0 & 0 \\
        0 & 0 & q_{2,2} & 1 
    \end{bmatrix}
    \Join
    {\widetilde Q}_3 
    =
    \begin{bmatrix}
        1 & q_{1,3}^6 & 1 & q_{2,3}^6\\
        q_{1,3}^2& 1 & 0 & 0 \\
        0 & 0 & q_{2,3}^2 & 1 
    \end{bmatrix}
    \Join
    \begin{bmatrix}
        I + q_{1,3} J' & J' \\
        q_{1,3} J & J + q_{1,3} I \\
        & & I + q_{2,3} J' & J' \\
        & & q_{2,3} J & J + q_{2,3} I
    \end{bmatrix}
    \\
    & = 
    \begin{bmatrix}
        I + q_{1,3} J' + q_{1,3}^7 J & J' + q_{1,3}^6 J + q_{1,3}^7 I & I + q_{2,3} J' + q_{2,3}^7 J & J' + q_{2,3}^6 J + q_{2,3}^7 I \\
        q_{1,3}^2 I + q_{1,3}^3 J' + q_{1,3} J  & q_{1,3}^2 J' + J + q_{1,3} I & & \\
        & & q_{2,3}^2 I + q_{2,3}^3 J' + q_{2,3} J  & q_{2,3}^2 J' + J + q_{2,3} I 
    \end{bmatrix}
    \\
    & = 
    \begin{bmatrix}
        I & J' + q_{1,3}^6 J & J' + q_{2,3}^6 J \\
        & q_{1,3} I + q_{1,3}^2 J' + J \\
        & & q_{2,3} I + q_{2,3}^2 J' + J
    \end{bmatrix}
    \Join
    \begin{bmatrix}
        1 & q_{1,3}^7 & 1 & q_{2,3}^7 \\
        q_{1,3} & 1 & 0 & 0 \\
        0 & 0 & q_{2,3} & 1
    \end{bmatrix}
    \equiv 
Q_3 \Join
    \begin{bmatrix}
        1 & q_{1,3}^7 & 1 & q_{2,3}^7 \\
        q_{1,3} & 1 & 0 & 0 \\
        0 & 0 & q_{2,3} & 1
    \end{bmatrix}.
\end{split} 
\]
By propagating it further, we obtain the following recurrence:
\[
    \begin{bmatrix}
        1 & q_{1,k-1}^{2^{k-1}-1} & 1 & q_{2,k-1}^{2^{k-1}-1} \\
        q_{1,k-1}& 1 & 0 & 0 \\
        0 & 0 & q_{2,k-1} & 1 
    \end{bmatrix}
    \Join
    {\widetilde Q}_k
    =
    Q_k \Join 
    \begin{bmatrix}
        1 & q_{1,k}^{2^{k}-1} & 1 & q_{2,k}^{2^{k}-1} \\
        q_{1,k}& 1 & 0 & 0 \\
        0 & 0 & q_{2,k} & 1 
    \end{bmatrix},
\]
where 
\[
    Q_k = 
    \begin{bmatrix}
        I & J' + q_{1,k}^{2^{k}-2} J & J' + q_{2,k}^{2^{k}-2} J \\
        & q_{1,k} I + q_{1,k}^2 J' + J \\
        & & q_{2,k} I + q_{2,k}^2 J' + J
    \end{bmatrix}.
\]
For the last core we obtain:
\[
\begin{split}
Q_\LL &=
    \begin{bmatrix}
        1 & q_{1,\LL-1}^{2^{\LL-1}-1} & 1 & q_{2,\LL-1}^{2^{\LL-1}-1} \\
        q_{1,\LL-1}& 1 & 0 & 0 \\
        0 & 0 & q_{2,\LL-1} & 1 
    \end{bmatrix}
    \Join
    {\widetilde Q}_\LL
    = 
    \begin{bmatrix}
        1 & q_{1,\LL}^{2^{\LL}-2} & 1 & q_{2,\LL}^{2^{\LL}-2} \\
        q_{1,\LL}^2 & 1 & 0 & 0 \\
        0 & 0 & q_{2,\LL}^2 & 1
    \end{bmatrix}
    \Join
    \begin{bmatrix}
        \alpha_1 \left(I + q_{1,\LL}J' \right) \\
        \alpha_1 q_{1,\LL} J \\
        \alpha_2 \left(I + q_{2,\LL}J' \right) \\
        \alpha_2 q_{2,\LL} J 
    \end{bmatrix}
    \\
    &= 
    \begin{bmatrix}
        (\alpha_1 + \alpha_2) I + (\alpha_1 q_{1,\LL} + \alpha_2 q_{2,\LL}) J' + (\alpha_1 q_{1,\LL}^{2^{\LL}-1}+ \alpha_2 q_{1,\LL}^{2^{\LL}-1}) J \\
        \alpha_1 q_{1,\LL} \left(J + q_{2,\LL}I + q_{1,\LL}^2 J' \right)\\
        \alpha_2 q_{2,\LL} \left(J + q_{2,\LL}I + q_{2,\LL}^2 J' \right)
    \end{bmatrix}
\end{split}
\]
\end{proof}

\section{Proofs of Section~\ref{sec:examples}} \label{app:examples}

	\begin{proof}[Proof of Proposition~\ref{prop:laur_prod}]
	Let us compute the first column of the circulant $C = AB$.
	\begin{align*}
	    C_{s,0}
	    &=
	    \sum_{t=0}^{N-1} A_{s,t}B_{t,0}
	    =
	    \sum_{t=0}^{N-1} A_{s-t \bmod N,0}~B_{t,0}
	    = \\ &=
	    \sum_{t=0}^{m_B-1} A_{s-t \bmod N,0}~B_{t,0}
	    +
	    \sum_{t=-n_B}^{-1} A_{s-(N+t) \bmod N,0}~B_{N+t,0}.
	\end{align*}
	If we denote $a_t$ and $b_t$ the coefficients of Laurent polynomials $f_A(z)$ and $f_B(z)$ respectively, we can simplify the above expression:
	\[
	C_{s, 0} = \sum_{t=-n_B}^{m_B-1} A_{s-t \bmod N,0}~b_t.
	\]
	
	Now we split the set of row indices $s$ into three parts:
	\begin{align*}
	S_1 &= \{0, \dots, m_A+m_B-2\},\\
	S_2 &= \{m_A+m_B-1, \dots, N-n_A-n_B-1\}, \\
	S_3 &=\{N-n_A-n_B,\dots,N-1\}.    
	\end{align*}
	If $s \in S_1$, then $(s-t) \in [-m_A, m_A+m_B-2+n_B]$ and we can write $A_{s-t \bmod N, 0} = a_{s-t}$ (here we imply that $a_i = 0$ for $i \not\in [-n_A, m_B-1]$).
	Next, if $s \in S_2$, then $(s-t) \in [m_A, N-n_A-1]$, thus $A_{s-t \bmod N,0} = 0$.
	Finally, if $s \in S_3$, then $(s-t) \in [N-n_A-n_B-m_B+1, N-1+n_B]$ and $A_{s-t \bmod N, 0} = a_{s-t-N}$.
	Taking the three cases together and using the notation $m = m_A + m_B$, $n=n_A + n_B$, we can write
	\begin{equation}\label{eq:C-s}
	C_{s,0} = \begin{cases}
	    \sum_{t=-n_B}^{m_B-1} a_{s-t}b_t,
	    &\text{ if } s \in [0, m-1],\\
	    0, &\text{ if } s \in [m, N-n-1], \\
	    \sum_{t=-n_B}^{m_B-1} a_{s-t-N}b_t, & \text{ if } s \in [N-n, N-1].
	\end{cases}
	\end{equation}
	So the circulant $C$ is indeed of the form~\eqref{eq:A} with parameters $m$ and $n$ (the properties $c_{m-1} \neq 0$ and $c_{-n} \neq 0$ will be checked a bit later).
	Moreover, as for $s \in [-n,-1]$ we have by definition
	$c_s = C_{N+s, 0}$, from~\eqref{eq:C-s} it follows that the general formula for $c_s$ is
	\[
	c_s = \sum_{t=-n_B}^{m_B-1} a_{s-t}b_t~~\text{ for all } s \in \{-n, \dots, m-1\}.
	\]
	This is exactly the formula for product of Laurent polynomials $f_A(z)$ and $f_B(z)$.
	
	Finally, we need to prove that $c_{m-1}$ and $c_{-n}$ are non-zero.
	By~\eqref{eq:C-s}, $c_{m-1} = a_{m_A-1}b_{m_B-1} \neq 0$ as $a_{m_A-1}$ and $b_{m_B-1}$ are non-zero.
	Similarly, $c_{-n} = a_{-n_A}b_{-n_B} \neq 0$.
	\end{proof}

\end{document}